\documentclass[11pt,reqno]{amsart}
\usepackage[tt=false]{libertine}
\usepackage{amssymb}
\usepackage{mathtools}
\usepackage[varbb]{newpxmath}
\usepackage[foot]{amsaddr}

\let\savedbigtimes\bigtimes
\let\bigtimes\relax
\usepackage{mathabx}
\let\bigtimes\savedbigtimes

\usepackage[margin=1in, bottom=1in]{geometry}
\usepackage{graphicx}
\usepackage{caption}
\usepackage{subcaption}
\usepackage{enumitem}
\usepackage{comment}
\usepackage{bbm}

\usepackage[usenames,dvipsnames]{xcolor}
\usepackage[colorlinks=true,
  linkcolor=teal!60!black,
  citecolor=PineGreen,
  urlcolor=RedViolet]{hyperref}
\hypersetup{bookmarksopen=true}

\usepackage[yyyymmdd,hhmmss]{datetime}

\usepackage{tikz}
\usetikzlibrary{decorations.pathreplacing,calligraphy}
\usetikzlibrary{calc}

\newtheorem{thm}{Theorem}[section]
\newtheorem{lem}[thm]{Lemma}
\newtheorem{ppn}[thm]{Proposition}
\newtheorem{cor}[thm]{Corollary}
\newtheorem{fac}[thm]{Fact}
\newtheorem{conj}[thm]{Conjecture}

\theoremstyle{definition}
\newtheorem{dfn}[thm]{Definition}

\theoremstyle{remark}
\newtheorem{rmk}[thm]{Remark}

\numberwithin{equation}{section}
\def\beq#1\eeq{%
    \begin{equation}%
    #1%
    \end{equation}%
}
\def\baln#1\ealn{%
    \begin{align}%
    #1%
    \end{align}%
}
\def\balnn#1\ealnn{%
    \begin{align*}%
    #1%
    \end{align*}%
}

\def\lt{\left}
\def\rt{\right}
\def\fr{\frac}
\def\la{\langle}
\def\ra{\rangle}
\def\eps{\varepsilon}

\def\bA{{\boldsymbol{A}}}
\def\bB{{\boldsymbol{B}}}

\def\bD{{\boldsymbol{D}}}
\def\bE{{\boldsymbol{E}}}

\def\bI{{\boldsymbol{I}}}

\def\bM{{\boldsymbol{M}}}

\def\bQ{{\boldsymbol{Q}}}

\def\bS{{\boldsymbol{S}}}

\def\bW{{\boldsymbol{W}}}
\def\bX{{\boldsymbol{X}}}

\def\bDelta{{\boldsymbol{\Delta}}}
\def\bLambda{{\boldsymbol{\Lambda}}}
\def\bomega{{\boldsymbol{\omega}}}

\def\be{{\boldsymbol{e}}}

\def\bg{{\boldsymbol{g}}}

\def\bu{{\boldsymbol{u}}}
\def\bv{{\boldsymbol{v}}}

\def\bx{{\boldsymbol{x}}}
\def\by{{\boldsymbol{y}}}

\def\bone{{\boldsymbol{1}}}

\def\bbC{{\mathbb{C}}}

\def\bbE{{\mathbb{E}}}

\def\bbP{{\mathbb{P}}}

\def\bbR{{\mathbb{R}}}
\def\bbS{{\mathbb{S}}}

\def\cE{{\mathcal{E}}}

\def\cG{{\mathcal{G}}}

\def\cN{{\mathcal{N}}}

\def\cQ{{\mathcal{Q}}}

\def\de{{\mathsf{d}}}
\def\FE{{\mathsf{FE}}}
\def\GOE{{\mathsf{GOE}}}
\def\lb{{\mathsf{lb}}}
\def\out{{\mathsf{out}}}
\def\polylog{{\mathsf{polylog}}}
\def\sc{{\mathsf{sc}}}
\def\sph{{\mathsf{sph}}}
\def\Sym{{\mathsf{Sym}}}
\def\Tr{{\mathsf{Tr}}}
\def\TW{{\mathsf{TW}}}
\def\ub{{\mathsf{ub}}}
\def\unif{{\mathsf{unif}}}
\def\Var{{\mathsf{Var}}}

\def\tZ{{\widetilde Z}}

\def\opsi{{\bar \psi}}

\def\oB{{\bar{\operatorname{B}}}}

\title{Fluctuations of the Sherrington--Kirkpatrick free energy at critical temperature}

\author{Hang Du \and Brice Huang}

\address[H.~Du]{Department of Mathematics, Massachusetts Institute of Technology}
\email[H.~Du]{hangdu@mit.edu}

\address[B.~Huang]{Department of Statistics, Stanford University}
\email[B.~Huang]{bmhuang@stanford.edu}

\date{\today}

\subjclass[2020]{Primary 82B44; Secondary 60K35, 60F05, 82B27}

\begin{document}

\begin{abstract}
  We consider the Sherrington--Kirkpatrick spin glass model at the critical inverse temperature $\beta = 1$ with zero external field.
  We prove that the free energy $F_N = F_{N,\beta=1}$ of this model has variance
  \[
    \Var(F_N) = \fr16 \log N + O(1)\,,
  \]
  confirming a physics prediction of Aspelmeier \cite{aspelmeier2008free}, and that the centered and scaled $F_N$ satisfies a Gaussian CLT.
  We also identify the critical two-replica overlap scale, proving
  \[
    \bbE \la R_{1,2}^2\ra \asymp N^{-2/3}\,,
  \]
  as conjectured by Talagrand \cite{talagrand2011mean2}, together with a uniform exponential moment bound for $N^{1/3} |R_{1,2}|$.
  The key input is a critical reweighted moment method, in the spirit of the ``small subgraph conditioning'' technique from probabilistic combinatorics, but capable of capturing diverging fluctuations.
  Through this reweighting, we relate the critical SK model to the BBP critical edge, which determines the overlap and fluctuation scales.
\end{abstract}

\maketitle

\setcounter{tocdepth}{1}
\tableofcontents

\section{Introduction and main results}
\label{s:intro}

The Sherrington--Kirkpatrick (SK) model was introduced in \cite{sherrington1975solvable} to model a disordered magnetic alloy with ``glassy'' behavior.
Formally, it is defined through the following random Hamiltonian $H_N$.
For positive integer $N$, define the state space $\Sigma_N = \{\pm 1\}^N$.
Let $\bW \sim \GOE(N)$; that is, $\bW \in \bbR^{N\times N}$ is symmetric with independent entries on and above the diagonal, with distribution $W_{i,i} \sim \cN(0, 2/N)$ and $W_{i,j} \sim \cN(0, 1/N)$ for $i<j$.
Then the SK model's Hamiltonian $H_N : \Sigma_N \rightarrow \bbR$ is given by
\beq\label{e:H}
  H_N(\bx) = \fr12 (\bW \bx, \bx)\,.
\eeq
Equivalently, $H_N$ is the Gaussian process on $\Sigma_N$ with covariance
\[
  \bbE H_N(\bx)H_N(\by) = \fr{N}{2} R(\bx,\by)^2\,,
\]
where $R(\bx,\by) = (\bx,\by)/N$ denotes the overlap of $\bx,\by$.
At inverse temperature $\beta \ge 0$, the partition function $Z_{N,\beta}$ and free energy $F_{N,\beta}$ of this model are defined by
\baln\label{e:Z-and-F}
  Z_{N,\beta} &= \fr{1}{2^N} \sum_{\bx \in \Sigma_N} e^{\beta H_N(\bx)}\,, & 
  F_{N,\beta} &= \log Z_{N,\beta}\,.
\ealn
Since the introduction of this model, a central question has been to characterize the in-probability limit of the free energy density $F_{N,\beta} / N$ as $N\rightarrow\infty$.
This limiting value was first predicted in the groundbreaking work of Parisi \cite{parisi1979infinite,parisi1983order}, and proved by Talagrand \cite{talagrand2006parisi} and Panchenko \cite{panchenko2013parisi} following decades of progress in the physics and probability communities \cite{mezard1987spin,ruelle1987mathematical,ghirlanda1998general,aizenman2003extended,guerra2003broken}.

Another natural question is to understand the fluctuations of the random variable $F_{N,\beta}$.
This question has attracted significant attention, in part because concentration and fluctuations of $F_{N,\beta}$ are closely tied to overlap structure, disorder chaos, and the geometry of the Gibbs measure \cite{chatterjee2009disorder,talagrand2010mean,talagrand2011mean2,chatterjee2014superconcentration,chen2019order}, defined by
\beq\label{e:gibbs-msr}
  \mu_{N,\beta}(\bx) = \fr{e^{\beta H_N(\bx)}}{2^N Z_{N,\beta}}\,.
\eeq
In the high-temperature phase $\beta < 1$, \cite{aizenman1987some,comets1995sherrington} showed that $F_{N,\beta}$ has Gaussian fluctuations of order $1$.
More precisely, they show\footnote{In the setting of \cite{aizenman1987some} the diagonal entries of the disorder $\bW$ are set to zero, which results instead in the limit distribution $\cN(-\tau^2, 2\tau^2)$ for $\tau^2 = -\fr14 (\log(1-\beta^2) + \beta^2)$.
Adapting their method to $\bW$ with diagonal entries gives the convergence in \eqref{e:ALR}.
The main results in Theorems~\ref{t:main} and \ref{t:overlap} remain the same under either convention, as the diagonal entries of $\bW$ contribute only an independent additive $O(1)$ to the free energy and do not affect the Gibbs measure.}
\baln\label{e:ALR}
  F_{N,\beta} - \fr{N\beta^2}{4} &\stackrel{d}{\rightarrow} \cN(-\sigma^2, 2\sigma^2)\,, &
  \sigma^2 &= -\fr14 \log(1-\beta^2)\,.
\ealn
The SK model has a phase transition at $\beta = 1$, beyond which the variance of $F_{N,\beta}$ is expected to diverge.
In the low-temperature phase $\beta > 1$, determining the scale of the fluctuations of $F_{N,\beta}$ was a significant open problem, resolved only recently in \cite{openai2026low} following progress in \cite{chatterjee2009disorder,aronow2026quantitative}.

The fluctuations of $F_{N,\beta}$ at the critical temperature $\beta=1$ are of particular interest, as they shed light on the model's phase transition at criticality.
Using non-rigorous physics methods, Aspelmeier \cite{aspelmeier2008free} (see also \cite{parisi2009phase}) predicted the following variance of the free energy.
\begin{conj}[\cite{aspelmeier2008free}]
  We have $\Var(F_{N,\beta=1}) = \fr16 \log N + O(1)$.
\end{conj}
The question of understanding the SK model at criticality was studied in detail in Talagrand's books \cite[Chapter~2]{talagrand2010mean} and \cite[Chapter~11]{talagrand2011mean2}.
He showed that for $\la \cdot \ra_\beta$ denoting average with respect to Gibbs samples $\bx^1,\bx^2$ sampled from \eqref{e:gibbs-msr}, and $R_{1,2} = R(\bx^1,\bx^2)$,
\beq\label{e:talagrand-R12}
  \bbE \la R_{1,2}^2 \ra_{\beta=1} = O(N^{-1/2})\,.
\eeq
As explained in \cite{chen2019order}, this implies the upper bound $\Var(F_{N,\beta=1}) = O(\sqrt{N})$.
Chen and Lam showed the following sharper bound.
\begin{thm}[{\cite[Theorem~1]{chen2019order}}]\label{t:CL19}
  We have $\Var(F_{N,\beta=1}) = O(\log^2 N)$.
\end{thm}
The recent work of Dey and Kang \cite{dey2026fluctuations} proved that for $\beta = 1 - CN^{-1/3}$ for any constant $C>0$, the variance prediction of \cite{aspelmeier2008free} holds:
\[
  \Var(F_{N,\beta}) = \fr16 \log N + O(1)\,.
\]
While in the final stages of writing this manuscript, we also learned of the concurrent work of Schertzer \cite{schertzer2026order}, which showed the bounds
\[
  \fr12 \log \log \log N - O(1) \le \Var(F_{N,\beta=1}) \le \fr14 \log N + O(1)\,.
\]
In this paper, we confirm the conjecture of \cite{aspelmeier2008free} at criticality.
\begin{thm}\label{t:main}
  The free energy $F_{N,\beta=1}$ of the critical SK model satisfies the following.
  \begin{enumerate}[label=(\alph*)]
    \item \label{i:main-var} $\Var(F_{N,\beta=1}) = \fr16 \log N + O(1)$. 
    \item \label{i:main-clt} The centered and rescaled $F_{N,\beta=1}$ satisfies a Gaussian CLT, with 
    \[
      \lt(\fr{\log N}{6}\rt)^{-1/2} \lt(F_{N,\beta=1} - \fr{N}{4} + \fr{\log N}{12} \rt)
      \stackrel{d}{\rightarrow}
      \cN(0,1)\,.
    \]
  \end{enumerate}
\end{thm}
\begin{rmk}
  The CLT in Theorem~\ref{t:main}\ref{i:main-clt} is universal in the disorder: by a standard Lindeberg exchange argument, the same conclusion holds if the Gaussian disorder is replaced by i.i.d. copies of any random variable $\xi$ satisfying $\bbE[\xi] = \bbE[\xi^3] = 0$, $\bbE[\xi^2] = 1$, $\bbE[\xi^4] < \infty$.
  However, this argument does not directly extend to Theorem~\ref{t:main}\ref{i:main-var} or Theorem~\ref{t:overlap} below.
  In this paper we focus on Gaussian disorder.
\end{rmk}
As part of his study of the critical phase transition, Talagrand \cite[Research~Problem~11.7.4]{talagrand2011mean2} asked to identify the order of $\bbE \la R_{1,2}^2 \ra_{\beta=1}$ and conjectured that it is $N^{-2/3}$.
As evidence for the lower bound of this conjecture, Chatterjee \cite[Proposition 11.7.6]{talagrand2011mean2} proved that for a universal $c>0$,
\[
  \bbE \la |R_{1,2}|^3 \ra_{\beta=1} \ge cN^{-1}\,.
\]
We confirm this conjecture.
For the upper bound, we in fact show an exponential moment bound at the critical scale $N^{-1/3}$.
\begin{thm}\label{t:overlap}
  There exists a universal constant $c>0$ such that the following holds.
  \begin{enumerate}[label=(\alph*)]
    \item \label{i:overlap-ub} $\bbE\la \exp(cN^{1/3}|R_{1,2}|)\ra_{\beta=1} \le 2$. 
    \item \label{i:overlap-lb} $\bbE\la R_{1,2}^2 \ra_{\beta=1} \ge cN^{-2/3}$.
  \end{enumerate}
  In particular, this implies $\bbE\la R_{1,2}^2 \ra_{\beta=1} \asymp N^{-2/3}$.
\end{thm}

\subsection{Related work}

A closely related but simpler model is the spherical SK model.
This model has Hamiltonian \eqref{e:H} on domain $S_N = \sqrt{N}\bbS^{N-1}$, with partition function and free energy
\baln\label{e:Z-and-F-sph}
  Z^\sph_{N,\beta} &= \int e^{\beta H_N(\bx)}\,\de \nu_N(\bx)\,, & 
  F^\sph_{N,\beta} &= \log Z^\sph_{N,\beta}\,,
\ealn
where $\nu_N$ denotes the uniform Haar measure on $S_N$.
The Parisi formula for the limiting free energy of this model (including for the more general mixed $p$-spin Hamiltonian) was established in \cite{talagrand2006spherical}.

For the spherical SK model, Baik and Lee \cite{baik2016fluctuations} showed that for $\beta < 1$, $F^\sph_{N,\beta}$ has Gaussian fluctuations of order $1$, with the same CLT \eqref{e:ALR}, and that for $\beta > 1$, $F^\sph_{N,\beta}$ has Tracy--Widom fluctuations of order $N^{1/3}$.
At criticality $\beta=1$, Landon \cite{landon2022free} showed $F^\sph_{N,\beta}$ has Gaussian fluctuations of order $\sqrt{\log N}$, with the same Gaussian CLT as Theorem~\ref{t:main}\ref{i:main-clt}.
For $\beta$ in the critical window $\beta = 1 + b N^{-1/3} \sqrt{\log N}$, \cite{landon2022free} (for $b \le 0$) and Johnstone, Klochkov, Onatski, and Pavlyshyn \cite{johnstone2024spin} (for $b\in \bbR$) showed that the limit becomes a sum of independent Gaussian and Tracy--Widom random variables:
\baln\label{e:spherical-critical-clt}
  \lt(\fr{\log N}{6}\rt)^{-1/2} \lt(F^\sph_{N,\beta} - N\FE(\beta) + \fr{\log N}{12} \rt)
  &\stackrel{d}{\rightarrow}
  \cN(0,1) + \sqrt{\fr{3}{2}} b_+ \TW_1\,, \\
  \nonumber
  \text{where} \qquad 
  \FE(\beta) &= \begin{cases}
    \beta^2 / 4 & \beta \le 1\,, \\
    \beta - \fr12 \log \beta - \fr34 & \beta > 1\,.
  \end{cases}
\ealn
Here $b_+ = \max(b,0)$.
All of these results rely on an exact contour integral formula (Lemma~\ref{lem:contour-integral-1}) for $Z^\sph_{N,\beta}$ introduced by \cite{baik2016fluctuations}, which can be analyzed precisely using random matrix theory.
As we explain in \S\ref{s:overview}, this random matrix interpretation of the spherical model will also be an important input to our proof.

The above methods have also been extended to study free energy fluctuations of multi-species spin glasses.
These include the limit distributions of the bipartite spherical SK model's free energy at all temperatures \cite{baik2020free,collins2025fluctuations} and multi-species Ising SK model at high temperature \cite{dey2021fluctuation}, and an upper bound on $\Var(F_{N,\beta})$ for the bipartite Ising SK model at criticality \cite{collins2025order}.

Finally, Prodromidis and Sly \cite{prodromidis2026distribution} recently characterized the fluctuations of the free energy, and limiting law of the magnetization, of the critical ferromagnetic Ising model on sparse $d$-regular and Erd\H{o}s--R\'enyi graphs.
See Remark~\ref{r:difference-with-ferro-ising} below for more on the relation between this paper and the present work.

\begin{rmk}
  We expect that for $\beta$ in the critical window $\beta = 1 + b N^{-1/3} \sqrt{\log N}$, the distributional limit \eqref{e:spherical-critical-clt} holds for the SK model as well.
  In Remark~\ref{r:critical-window-extension} we describe a strategy for proving this.
  This strategy involves random matrix inputs that are much more delicate than what the present paper requires, and we do not pursue this here.
  After the initial posting of this paper, this approach was carried out in \cite{cheng2026critical}.
\end{rmk}

\subsection{Main idea: spherical SK as critical reweighting}
\label{ss:critical-reweighting}

Our contribution can be interpreted in two ways.
\begin{enumerate}[label=(\roman*)]
  \item \label{i:critical-reweighting-comparison}
  \textbf{Sphere to cube comparison}: we show that $Z^\sph_{N,\beta=1}$ explains the fluctuations of $Z_{N,\beta=1}$ up to a stochastic $1+o_N(1)$ factor.
  This is proved by a reweighted second moment argument in the spirit of the ``small subgraph conditioning'' method \cite{robinson1992almost,robinson1994almost}, and is the first such argument where the explanatory variable has diverging fluctuations.
  \item \label{i:critical-reweighting-overlap}
  \textbf{Critical reweighting exposes the sharp overlap scale}: the $N^{-2/3}$ eigenvalue spacing at the GOE spectral edge leads to the prediction that in the critical SK model, $R_{1,2}$ is of scale $N^{-1/3}$.
  This gives rise to both the exponent $-2/3$ in Theorem~\ref{t:overlap} and factor $1/6$ in Theorem~\ref{t:main}\ref{i:main-var}.
  A key challenge in proving these theorems is that a direct moment-based overlap analysis does not capture the sharp overlap scale.
  We show that reweighting by $Z^\sph_{N,\beta=1}$ leads to a moment calculation that sees the optimal scale $R_{1,2} \asymp N^{-1/3}$, leading to the proofs of these theorems.
\end{enumerate}

Both perspectives are used in this paper.
The CLT in Theorem~\ref{t:main}\ref{i:main-clt} is proved by comparison with the critical spherical SK model, whose free energy CLT \eqref{e:spherical-critical-clt} is already known.
However, we emphasize that the sharp variance estimate in Theorem~\ref{t:main}\ref{i:main-var} and overlap bounds in Theorem~\ref{t:overlap} are \textbf{not} proved by transferring known results about the spherical model.
Indeed, the spherical analogues of these results were not previously known:  while the CLT \eqref{e:spherical-critical-clt} suggests $\Var(F^\sph_{N,\beta=1})$ has leading order $\fr16 \log N$, this does not formally follow, nor is the precision $O(1)$ in Theorem~\ref{t:main}\ref{i:main-var} clear.
We instead prove Theorems~\ref{t:main}\ref{i:main-var} and~\ref{t:overlap} through the overlap analysis in perspective \ref{i:critical-reweighting-overlap}.
We next explain these two perspectives.

\subsubsection{Small subgraph conditioning and sphere to cube comparison}

The SK model is well understood in the high-temperature phase $\beta < 1$.
The main difficulty at criticality is that as $\beta \uparrow 1$, the second moment ratio $\bbE[Z_{N,\beta}^2] / \bbE[Z_{N,\beta}]^2$ diverges.
Indeed, a simple calculation shows that for fixed $\beta < 1$,
\beq\label{e:high-temp-2mt-ratio}
  \fr{\bbE[Z_{N,\beta}^2]}{\bbE[Z_{N,\beta}]^2}
  = 4^{-N}\sum_{\bx,\by \in \Sigma_N} e^{N \beta^2 R(\bx,\by)^2 / 2}
  = 2^{-N} \sum_{q\in \{-1,-1+\fr2N,\ldots,1\}} \binom{N}{\fr{1+q}{2} \cdot N} e^{N\beta^2 q^2/2}
  \stackrel{N\rightarrow\infty}{\longrightarrow} \fr{1}{\sqrt{1-\beta^2}}
\eeq
is bounded independently of $N$.
In such settings, $Z_{N,\beta} / \bbE[Z_{N,\beta}]$ has $O(1)$ fluctuations, and there are standard methods (described below) to identify these fluctuations.
However, as $\beta \uparrow 1$ the right-hand side of \eqref{e:high-temp-2mt-ratio} diverges, reflecting diverging quenched fluctuations, and such methods no longer directly apply.

\textbf{Small subgraph conditioning} is a powerful method to identify fluctuations of the partition function $\tZ_N$ of a general ``high-temperature'' spin system where $\tZ_N / \bbE[\tZ_N]$ has $O(1)$ quenched fluctuations.
The idea is to identify a reweighting random variable $A_N$, whose distribution we know, such that\footnote{In some models, the main contribution to $\bbE[\tZ_N^2]$ comes from a large deviation event.
In such cases one aims to identify $A_N$ such that \eqref{e:ssc-goal} holds after further truncating $\tZ_N$ to a typical event.
For the same reason, the quenched fluctuations of $\tZ_N / \bbE[\tZ_N]$ do not always diverge when $\bbE[\tZ_N^2] / \bbE[\tZ_N]^2$ does, though in our model this will be the case.}
\beq\label{e:ssc-goal}
  \bbE[(\tZ_N / A_N - 1)^2] \ll 1\,.
\eeq
That is, $A_N$ ``explains the fluctuations'' of $\tZ_N$.
This method has been applied to many spin systems on random graphs \cite{robinson1992almost,robinson1994almost,cooper1996perfect,mossel2009hardness,kemkes2010chromatic,fabian2021ising,coja2026fluctuations,prodromidis2026distribution} and mean-field models including the SK model \cite{abbe2022proof,bencs2025zeros}.
The name ``small subgraph conditioning'' comes from the fact that in random graph settings, one takes $A_N = \bbE[\tZ_N] A'_N$ where $A'_N$ is a statistic of subgraphs of bounded or slowly growing size.
In mean-field settings, the analogous $A'_N$ is a low-degree cluster expansion in $\bW$.

However, the assumption that $\tZ_N / \bbE[\tZ_N]$ has $O(1)$ quenched fluctuations was essential to previous applications of this method.
This is because a small subgraph-based statistic $A_N$ can only explain a constant amount of variance; if this assumption fails the amount of variance $A_N$ needs to explain must also diverge.

We introduce a critical version of this reweighted second moment idea, where the explanatory variable $A_N$ is the \emph{spherical} SK partition function $Z^\sph_{N,\beta=1}$.
This has diverging fluctuations and remains an effective approximation of $Z_{N,\beta=1}$ even at criticality.
At the same time, the spherical model's aforementioned random matrix interpretation makes the resulting moment calculation tractable.
We will show:
\begin{thm}\label{t:compare-cube-sphere}
  Let $X_N = Z_{N,\beta=1} / Z^\sph_{N,\beta=1}$.
  Then, $\bbE[(X_N-1)^2] \lesssim N^{-1/3}$.
\end{thm}

\subsubsection{Critical reweighting exposes the sharp overlap scale}

Moment calculations similar to \eqref{e:high-temp-2mt-ratio} also provide a way to understand overlaps.
In earlier work \cite[Lemma 11.7.10]{talagrand2011mean2} one way to estimate an overlap moment $\bbE \la g(R_{1,2}) \ra$, for test function $g$, was to control the partition function-weighted overlap moment
\beq\label{e:naive-weighted-mt}
  \bbE [Z_N^2 \la g(R_{1,2}) \ra]
  = \fr{1}{4^N} \sum_{\bx,\by \in \Sigma_N} g(R(\bx,\by)) \bbE[e^{H_N(\bx) + H_N(\by)}]
  = 2^{-N} \sum_q g(q) \binom{N}{\fr{1+q}{2}N} e^{Nq^2/2}\,,
\eeq
which can be explicitly calculated.
This approach is effective in the high-temperature regime $\beta<1$, see \cite[Chapter 11]{talagrand2011mean2}.
However, as $\beta \uparrow 1$, bounds on \eqref{e:naive-weighted-mt} do not control $\bbE \la g(R_{1,2}) \ra$ as effectively.
One way to see the problem at $\beta = 1$ is that, for (e.g.) polynomial $g$, the main contribution to \eqref{e:naive-weighted-mt} comes from $|R_{1,2}| \lesssim N^{-1/4}$ (leading to the result \eqref{e:talagrand-R12}), but to prove Theorem~\ref{t:overlap} we need $|R_{1,2}|$ to be of scale $N^{-1/3}$.

Reweighting by the spherical partition function solves this problem.
For the $g$ appearing in Theorem~\ref{t:overlap}, we will estimate the reweighted overlap moment
\beq\label{e:reweighted-overlap-mt}
  \bbE[X_N^2 \la g(R_{1,2}) \ra]\,,
\eeq
A key point of our analysis is that \eqref{e:reweighted-overlap-mt} remains tractable (see Lemma~\ref{l:reweighted-2mt}), and its main contribution comes from $|R_{1,2}| \lesssim N^{-1/3}$.
\textbf{As we explain in \S\ref{s:overview}, the exponent $-1/3$ formally enters our proof through the BBP edge scaling \cite{baik2005phase,peche2006largest,bloemendal2013limits,bloemendal2016limits}, which appears in these moment calculations because we have reweighted by $Z^\sph_{N,\beta=1}$.}
These ideas give control of \eqref{e:reweighted-overlap-mt}.
Together with control of the tails of $X_N$, this implies Theorem~\ref{t:overlap}.
Theorem~\ref{t:main}\ref{i:main-var} follows from Theorem~\ref{t:overlap} through a known connection between overlaps and the free energy variance \cite{chatterjee2009disorder,chen2019order,dey2026fluctuations}.

\subsubsection{Further consequences}

We also derive the spherical analogues of Theorems~\ref{t:main}\ref{i:main-var} and~\ref{t:overlap} from their Ising versions.
Let $\la \cdot \ra_\beta^\sph$ denote average with respect to the Gibbs measure on $S_N$ with density
\[
  \de \mu^\sph_{N,\beta}(\bx) = \fr{e^{\beta H_N(\bx)}}{Z^\sph_{N,\beta}}\,\de \nu_N(\bx)\,.
\]
\begin{cor}\label{c:sphere}
  There exists a universal constant $c>0$ such that the following holds.
  \begin{enumerate}[label=(\alph*)]
    \item \label{i:sphere-var} $\Var(F^\sph_{N,\beta=1}) = \fr16 \log N + O(1)$.
    \item \label{i:sphere-overlap-ub} $\bbE \la \exp(cN^{1/3} |R_{1,2}|) \ra^\sph_{\beta=1} \le 2$.
    \item \label{i:sphere-overlap-lb} $\bbE \la R_{1,2}^2 \ra^\sph_{\beta=1} \ge cN^{-2/3}$.
  \end{enumerate}
\end{cor}

Since the initial posting of this paper, subsequent works have demonstrated that the ideas \ref{i:critical-reweighting-comparison}--\ref{i:critical-reweighting-overlap} are more broadly useful in studying mean-field disordered systems' behavior at and near criticality:
\begin{itemize}
  \item In the companion paper \cite{du2026overlap}, we use a strengthened sphere to cube comparison to characterize the limiting quenched distribution of $N^{1/3} R_{1,2}$ in the critical SK model.
  \item \cite{cheng2026critical} shows that the comparison \ref{i:critical-reweighting-comparison} continues to hold for $\beta$ in the critical window $\beta = 1 + b N^{-1/3} \sqrt{\log N}$.
  This characterizes the limiting fluctuations of the SK free energy throughout the critical window and shows it coincides with the spherical SK model's limit \eqref{e:spherical-critical-clt}.
  \item In joint work with Gong and Li \cite{du2026critical}, we determine the free energy fluctuations of a family of critical perceptron models, and show they also satisfy a CLT on the universal scale $(\frac{\log N}{6})^{1/2}$.
  There, the main idea is to compare both spherical and Ising perceptron models with a critical spherical Hopfield model.
  We compare to this model because it, like the spherical SK model, is a spherical model with a quadratic random matrix interaction.
  Its fluctuation scale $(\frac{\log N}{6})^{1/2}$ comes from the $N^{-2/3}$ edge spacing of the Wishart ensemble, and can be established similarly to \cite{landon2022free}.
  Thus our ideas go well beyond comparing an Ising spin glass to its spherical analogue, and provide a general connection between spin glasses' critical behavior and random matrices' edge behavior.
\end{itemize}

\begin{rmk}[{Relation to \cite{prodromidis2026distribution}}]\label{r:difference-with-ferro-ising}
  The recent paper \cite{prodromidis2026distribution} also uses small subgraph conditioning to characterize the fluctuations of the critical ferromagnetic Ising model.
  The mechanisms driving their model's critical transition are different from those of our model, and we view these two papers as complementary explorations of different types of critical transitions.

  In their ferromagnetic model, one can write the partition function $\tZ_N$ as a sum of contributions $\tZ_{N,m}$ from each magnetization $m = (\bx, \bone)$.
  This model's critical transition is driven by $\tZ_{N,m=0}$ changing from a global maximum to a local minimum in $m$ as the temperature approaches criticality.
  At the same time, the fixed-magnetization models $\tZ_{N,m}$ (for the relevant $m$ near $0$) have second moment ratios $\bbE[\tZ_{N,m}^2] / \bbE[\tZ_{N,m}]^2$ that remain bounded.
  In contrast, our model's critical transition is driven by $\bbE[Z_{N,\beta}^2] / \bbE[Z_{N,\beta}]^2$ diverging as $\beta \uparrow 1$, with quenched fluctuations that also diverge.
  While both models require factors $A_N$ to explain the critical fluctuations that go beyond the usual small subgraph conditioning methodology, the nature of the factors $A_N$ used is accordingly different.
\end{rmk}

\subsection*{Notation}

\textbf{For the rest of the paper we set $\beta = 1$.}
We will abbreviate $Z_N = Z_{N,\beta=1}$, and similarly $F_N$, $Z^\sph_N$, $F^\sph_N$.
We will sometimes write $Z_N(\bW)$ (and so on) to emphasize the dependence of $Z_N$ on $\bW$.

We will use $\bbE$ to denote expectation with respect to the disorder $\bW$ and $\la \cdot \ra$, $\la \cdot \ra^\sph$ (and other variants we will introduce) to denote Gibbs average with respect to \eqref{e:gibbs-msr}, \eqref{e:Z-and-F-sph} conditional on $\bW$.
The Gibbs averages will always be with respect to inverse temperature $\beta=1$.
There will be no confusion between $\la \cdot \ra$ and the Euclidean inner product, which will be denoted $(\cdot, \cdot)$.
The Frobenius inner product and norm are denoted $(\bA,\bB)_F = \Tr(\bA\bB)$ and $\|\bA\|_F^2 = (\bA,\bA)_F = \Tr(\bA^2)$ for symmetric matrices $\bA,\bB \in \bbR^{N\times N}$.

We will use boldface symbols for vector- and matrix-valued variables in $\bbR^N$ and $\bbR^{N\times N}$, and plain symbols for finite-dimensional variables.

We use standard asymptotic notation: $f = O(g)$, $g = \Omega(f)$, $f\lesssim g$ all mean that $f \le Cg$ for a universal constant $C$; $f = \Theta(g)$ and $f \asymp g$ mean $f \lesssim g \lesssim f$; and $f = o(g)$, $f \ll g$ mean $f/g \rightarrow 0$ as $N\rightarrow\infty$.
All estimates are for $N$ sufficiently large.
Throughout, $c,C$ denote small and large universal constants that may change from line to line.
We take $\eps = 0.01$ to be a small explicit constant.

\subsection*{Acknowledgements}

Part of this work was completed while the authors were visiting Peking University.
We are grateful to Sourav Chatterjee, Wei-Kuo Chen, Jason Prodromidis, and Mark Sellke for helpful feedback on the manuscript, and Shuyang Gong, Zhangsong Li, Kevin Luo, Jason Prodromidis, Tselil Schramm, and Mark Sellke for motivating conversations.
HD was partially supported by an NSF-Simons research collaboration grant (award number 2031883).
BH was supported by a Stanford Science Fellowship and an NSF Mathematical Sciences Postdoctoral Fellowship.

\subsection*{Statement on AI use}

The overall strategy and architecture of the proof, as described in \S\ref{s:intro}--\ref{s:overview} was developed entirely by the human authors. 
AI assistance was limited to the proofs of Propositions~\ref{ppn:denominator-upper-bound} and \ref{prop:numerator-lower-bound}.
We were aware of the characteristic function interpretation of the contour integral appearing in these proofs, and expected these propositions to follow from it and standard GOE critical-edge estimates.
We used GPT 5.5 Pro to locate the relevant references and assemble the proofs of these propositions, which we then digested and simplified.
All other proofs in the paper were human generated.
The authors take full responsibility for the correctness of the paper.

\section{Proof overview}
\label{s:overview}

In \S\ref{ss:critical-reweighting} we explained that our proof is based on estimating second moments reweighted by $Z^\sph_N$.
In this section, we explain how these reweighted moments are computed, why the critical overlap scale is $N^{-1/3}$, and how the resulting estimates yield Theorems~\ref{t:main}--\ref{t:overlap} and \ref{t:compare-cube-sphere}.
We focus on the estimates:
\baln\label{e:overview-goal-compare}
   \bbE[(X_N-1)^2] \lesssim N^{-1/3}\,,\qquad\qquad\qquad\ \ \\
\label{e:overview-goal-overlap}
  \bbE \la R_{1,2}^2 \ra \lesssim N^{-2/3}\,, \qquad\qquad\qquad\qquad
  \bbE \la R_{1,2}^4 \ra \lesssim N^{-4/3}\,,
\ealn
where $X_N = Z_N / Z^\sph_N$.
The first estimate is Theorem~\ref{t:compare-cube-sphere}.
The second pair is a weaker version of Theorem~\ref{t:overlap}\ref{i:overlap-ub}, and the exponential moment bound in Theorem~\ref{t:overlap}\ref{i:overlap-ub} will follow from similar ideas as it.

It is a well-established fact in the free energy fluctuations literature \cite{chatterjee2009disorder,chen2019order,dey2026fluctuations} that suitably sharp upper bounds on annealed overlap moments imply (two-sided) bounds on the free energy variance.
In particular, the method of \cite{dey2026fluctuations} (building on \cite{chatterjee2009disorder,talagrand2011mean2}) shows that \eqref{e:overview-goal-overlap} implies Theorem~\ref{t:main}\ref{i:main-var}, the desired estimate on $\Var(F_N)$ at criticality.
This section is organized as follows.
\begin{itemize}
    \item In \S\ref{ss:reweighted-overlap}, we state Lemma~\ref{l:reweighted-2mt}, the key reweighted overlap identity, which expresses $\bbE[X_N^2 \la g(R_{1,2}) \ra]$ as a one-dimensional expectation over the overlap $q$ of a function $J(q)$.
    We also explain how \eqref{e:overview-goal-compare}, \eqref{e:overview-goal-overlap} both reduce to estimating reweighted moments of this form.

    \item In \S\ref{ss:dominant-contribution}, we explain how localization of this one-dimensional integrand on the scale $|q|\lesssim N^{-1/3}$ implies \eqref{e:overview-goal-compare} and a reweighted version \eqref{e:overview-goal-overlap-reweight} of \eqref{e:overview-goal-overlap}, from which \eqref{e:overview-goal-overlap} follows.

    \item In \S\ref{ss:bbp}, we explain why the localization scale is $|q| \lesssim N^{-1/3}$.
    Heuristically, this comes from the BBP edge transition.
    The formal proof avoids a full BBP computation, and instead uses log-concavity considerations to reduce to a one-point lower bound on $J(0)$.

    \item In \S\ref{ss:organization}, we summarize the organization of the rest of the paper.
\end{itemize}

\subsection{A one-dimensional identity for reweighted overlaps}
\label{ss:reweighted-overlap}

We next explain how the goals \eqref{e:overview-goal-compare}, \eqref{e:overview-goal-overlap} reduce to estimating reweighted overlap moments of the form $\bbE [X_N^2 \la g(R_{1,2})\ra]$, and then present Lemma~\ref{l:reweighted-2mt}, our central algebraic identity for these reweighted moments.
First, note that $Z^\sph_N(\bW)$ is an average of $Z_N(\bW)$ over orthogonal rotations of $\bW$, and thus $\bbE[X_N] = 1$.
So,
\beq\label{e:overview-goal-compare-step1}
  \bbE[(X_N-1)^2] = \bbE[X_N^2] - 1\,,
\eeq
and \eqref{e:overview-goal-compare} reduces to showing $\bbE[X_N^2] = 1 + O(N^{-1/3})$.
The goal \eqref{e:overview-goal-overlap} essentially follows from
\baln\label{e:overview-goal-overlap-reweight}
  \bbE [X_N^2 \la R_{1,2}^2 \ra] &\lesssim N^{-2/3}\,, &
  \bbE [X_N^2 \la R_{1,2}^4 \ra] &\lesssim N^{-4/3}
\ealn
as the lower tail of $X_N$ can be controlled.
The reweighting by $X_N^2$ makes the overlap moments in \eqref{e:overview-goal-overlap-reweight} tractable, thanks to an explicit one-dimensional representation given by the following lemma.
This is proved in \S\ref{ss:K-properties} by rotational invariance and Gaussian change of measure.
\begin{lem}
  \label{l:reweighted-2mt}
  For $\bx,\by \in S_N$ with $R(\bx,\by) = q$, define
  \beq\label{e:def-J}
    J(q) = \bbE\lt[\fr{e^{H_N(\bx) + H_N(\by)}}{(Z^\sph_N)^2}\rt]\,.
  \eeq
  Note that by rotational invariance of $\bW$, this expectation depends on $\bx,\by$ through their overlap $q = R(\bx,\by)$, so the notation $J(q)$ is justified.
  Then, for $\be_1,\be_2$ the first two (unit) standard basis vectors,
  \baln\label{e:K}
    J(q) &= e^{Nq^2/2} \cdot e^{N/2} K(q)\,, & 
    &\text{where}&
    K(q) &= \bbE\lt[Z^\sph_N\lt(\bW + (1+q) \be_1\be_1^\top + (1-q) \be_2\be_2^\top\rt)^{-2}\rt]\,,
  \ealn
  For $\bbE_q$ denoting expectation over $q = R(\bx,\by)$ for i.i.d. samples $\bx,\by \sim \unif(\Sigma_N)$, and measurable $g$,
  \beq\label{e:reweighted-2mt}
    \bbE[X_N^2 \la g(R_{1,2}) \ra] = \bbE_q [J(q) g(q)] = \bbE_q[e^{Nq^2/2} \cdot e^{N/2}K(q) \cdot g(q)]\,.
  \eeq
  Finally, for $\bbE^\sph_q$ denoting expectation over $q = R(\bx,\by)$ for i.i.d. samples $\bx,\by \sim \unif(S_N)$,
  \beq\label{e:reweighted-sph-2mt}
    \bbE[\la g(R_{1,2}) \ra^\sph] = \bbE^\sph_q [J(q) g(q)] = \bbE^\sph_q[e^{Nq^2/2} \cdot e^{N/2}K(q) \cdot g(q)]\,.
  \eeq
\end{lem}
\begin{rmk}
  The factor $e^{Nq^2/2}$ in \eqref{e:reweighted-2mt} will cancel the Gaussian curvature of $\bbE_q$ near $0$, and is analogous to the factor $e^{N\beta^2q^2/2}$ in \eqref{e:high-temp-2mt-ratio}.
  Compared to \eqref{e:high-temp-2mt-ratio}, the additional factor $K(q)$ prevents \eqref{e:reweighted-2mt} from diverging.
\end{rmk}

\subsection{Optimal-scale localization implies the critical exponents}
\label{ss:dominant-contribution}

In \S\ref{ss:bbp}, we explain how the function $K(q)$ exponentially decays as $q$ varies away from $0$, making the main contribution to \eqref{e:reweighted-2mt} for any polynomial $g$ come from scale $|q| \lesssim N^{-1/3}$.
In this subsection, we first explain how this leads to the scalings in \eqref{e:overview-goal-compare}, \eqref{e:overview-goal-overlap-reweight}.
Taking $g=1$ in \eqref{e:reweighted-2mt}, \eqref{e:reweighted-sph-2mt} gives
\baln\label{e:compare-XN2-to-1}
  \bbE[X_N^2] &= \bbE_q[e^{Nq^2/2} \cdot e^{N/2}K(q)]\,, &
  1 &= \bbE^\sph_q[e^{Nq^2/2} \cdot e^{N/2}K(q)]\,.
\ealn
Showing $\bbE[X_N^2] = 1 + O(N^{-1/3})$ (which gives \eqref{e:overview-goal-compare}) reduces to showing the expectations in \eqref{e:compare-XN2-to-1} differ by $O(N^{-1/3})$. 
Under $\bbE_q$, $q$ is sampled from a discrete probability measure on $\cQ_N = \{-1,-1+\fr{2}{N},\ldots,1-\fr{2}{N},1\}$ with mass
\beq\label{e:ising-mass}
  p(q) = 2^{-N} \binom{N}{N\cdot \fr{1+q}{2}}
  \,\,\propto\,\, \exp\lt(-\fr{N}{2} q^2 + O(q^2) - \Theta(Nq^4)\rt)\,.
\eeq
Under $\bbE^\sph_q$, $q$ is sampled from the probability measure on $[-1,1]$ with density
\beq\label{e:sphere-density}
  \rho(q) = \fr{\Gamma(N/2)}{\sqrt{\pi}\Gamma((N-1)/2)}  (1-q^2)^{(N-3)/2}
  \,\,\propto\,\, \exp\lt(-\fr{N}{2} q^2 + O(q^2) - \Theta(Nq^4)\rt)\,.
\eeq
At the effective scale $|q| \lesssim N^{-1/3}$, the error terms in \eqref{e:ising-mass}, \eqref{e:sphere-density} are each $O(N^{-1/3})$.
As $K$ is regular enough to compare the discrete and continuous expectations, this implies $\bbE[X_N^2] = 1 + O(N^{-1/3})$, which shows \eqref{e:overview-goal-compare}.
Furthermore, \eqref{e:reweighted-2mt} implies
\baln
  \bbE[X_N^2] &= \bbE_q[J(q)]\,, & 
  \bbE[X_N^2 \la R_{1,2}^{2k} \ra] &= \bbE_q[J(q) q^{2k}]\,.
\ealn
Since $\bbE[X_N^2] \asymp 1$, and the main contribution to these expectations comes from $|q| \lesssim N^{-1/3}$, \eqref{e:overview-goal-overlap-reweight} follows.

\subsection{Why localization occurs at the BBP scale}
\label{ss:bbp}

The scaling $N^{-1/3}$ comes from the BBP edge transition.
In the definition \eqref{e:K} of $K$, the matrix
\[
  \bW + (1+q) \be_1\be_1^\top + (1-q) \be_2\be_2^\top
\]
is a GOE matrix with two spikes, which are exactly critical when $q=0$.
As we vary $q$ away from $0$, the larger spike becomes supercritical, and the edge eigenvalue process enters the BBP critical window when $|q|\asymp N^{-1/3}$ \cite{baik2005phase,peche2006largest,bloemendal2013limits,bloemendal2016limits}.
From such considerations, and the aforementioned contour integral formula for $Z^\sph_N$ (Lemma~\ref{lem:contour-integral-1}), we expect (see Remark~\ref{r:heuristic-K} below for a heuristic justification)
\beq\label{e:K-true-asymptotic}
  \log K(q) \approx \log K(0) - CN|q|^3 + \text{lower order}\,.
\eeq
Since the $e^{Nq^2/2}$ in \eqref{e:reweighted-2mt} exactly cancels the leading $\exp(-\fr{N}{2}q^2)$ in \eqref{e:ising-mass}, the decay rate \eqref{e:K-true-asymptotic} of $K$ ensures the main contribution to \eqref{e:reweighted-2mt} is from $|q| \lesssim N^{-1/3}$.

Proving \eqref{e:K-true-asymptotic} amounts to an explicit, though delicate, random matrix calculation that we do not attempt in this paper.
In our formal proof, we use a softer strategy to show the main contribution to \eqref{e:reweighted-2mt} comes from $|q| \lesssim N^{-1/3}$, in order to reduce the random matrix inputs required.
Our proof is based on showing:
\begin{enumerate}[label=(\roman*)]
  \item \label{i:J-of-zero} $J(0) \gtrsim N^{-1/6}$ (Proposition~\ref{p:J-of-zero}); and 
  \item \label{i:K-log-concave} $K$ is even and log-concave (Lemma~\ref{l:K-log-concave}).
\end{enumerate}
The point is that for $\rho(q)$ the spherical density defined in \eqref{e:sphere-density}, the function
\[
  \psi(q) = \rho(q) J(q) = \rho(q) \cdot e^{Nq^2/2} \cdot e^{N/2} K(q)
\]
integrates to $1$ by \eqref{e:compare-XN2-to-1}, and is essentially log-concave by \ref{i:K-log-concave}.
Since $\rho(0) \asymp N^{1/2}$, the lower bound \ref{i:J-of-zero} implies $\psi(0) \gtrsim N^{1/3}$.
An even log-concave probability density with central value $h$ has tails at scale $h^{-1}$, which implies $\psi$ is localized on $|q| \lesssim N^{-1/3}$.

Input \ref{i:J-of-zero} is the only place in the proof where detailed random matrix estimates enter.
However, the random matrix theory needed is far simpler than for a complete proof of the BBP asymptotic \eqref{e:K-true-asymptotic}, as we just need a one-point estimate at $q=0$.
The fact that we only need a lower bound on $J(0)$ also simplifies the proof, as we do not need to control low-probability contributions to the expectation \eqref{e:def-J} defining $J$.
\begin{rmk}\label{r:heuristic-K}
  One may predict \eqref{e:K-true-asymptotic} from the following heuristic calculation.
  If $\bM(q) = \bW + (1+q) \be_1\be_1^\top + (1-q) \be_2\be_2^\top$ has eigenvalues $\lambda_1 \ge \cdots \ge \lambda_N$, a saddle-point analysis of the contour integral formula in Lemma~\ref{lem:contour-integral-1} (\cite[Lemma 1.3]{baik2016fluctuations}) suggests
  \beq\label{e:heuristic-K-1}
    \fr1N \log (Z_N^\sph(\bM(q))^{-2}) = 1 + \sup_{\gamma > \lambda_1} \lt\{
      \fr1N\sum_{k=1}^N\log(\gamma-\lambda_k) - \gamma
    \rt\} + o_N(1)\,.
  \eeq
  We make two heuristic assumptions: first, the spectrum of $\bM(q)$ can be approximated by $\fr1N \delta(\lambda_{\out}(q)) + \fr{N-1}{N} \rho_{\sc}$, where $\lambda_{\out}(q) = (1+q) + \frac{1}{1+q}$ is the outlier eigenvalue predicted by the BBP phase transition, $\delta(\lambda_{\out}(q))$ denotes the Dirac delta with this value, and $\de \rho_{\sc}(\lambda) = \fr{1}{2\pi} \sqrt{4-\lambda^2} \,\de \lambda$ is the semicircle law; and second, $Z_N^\sph(\bM(q))^{-2} \approx K(q)$ is self-averaging.
  Then, \eqref{e:heuristic-K-1} suggests
  \[
    \fr1N \log K(q) \approx 1 + \sup_{\gamma > \lambda_{\out}(q)} \lt\{
      \fr1N \log (\gamma - 2 - q^2) + \fr{N-1}{N} \int \log(\gamma - \lambda)\,\de \rho_{\sc}(\lambda)
      - \gamma
    \rt\}\,.
  \]
  If we write $\kappa(q)$ for this supremum, a direct calculation shows $\kappa(0) - \kappa(q) \asymp |q|^3$, which suggests \eqref{e:K-true-asymptotic}.
\end{rmk}
\begin{rmk}\label{r:critical-window-extension}
  Consider $\beta$ in the critical window $\beta = 1 + b N^{-1/3} \sqrt{\log N}$, and let $X_{N,\beta} = Z_{N,\beta} / Z^\sph_{N,\beta}$.
  We outline here what is needed to extend the distributional limit \eqref{e:spherical-critical-clt} to the SK model by proving
  \beq\label{e:critical-window-extension-goal}
    \bbE[(X_{N,\beta}-1)^2] \ll 1\,.
  \eeq
  For analogous $J_\beta(q)$, $K_\beta(q)$, we have similarly to \eqref{e:overview-goal-compare-step1}, \eqref{e:compare-XN2-to-1}
  \beq\label{e:critical-window-extension}
    \bbE[(X_{N,\beta}-1)^2]
    = \bbE[X_{N,\beta}^2] - 1
    = \bbE_q[e^{N\beta^2 q^2/2} \cdot e^{N\beta^2/2} K_\beta(q)] - \bbE^\sph_q[e^{N\beta^2 q^2/2} \cdot e^{N\beta^2 /2} K_\beta(q)]\,.
  \eeq
  This is $o(1)$ provided the main contribution to \eqref{e:critical-window-extension} is from $|q| \ll N^{-1/4}$.
  For $b \le 0$, this can be shown by the same strategy as above.
  However, when $b>0$ the proof strategy based on \ref{i:J-of-zero}--\ref{i:K-log-concave} no longer works: the leading $\exp(-\fr{N}{2}q^2)$ in \eqref{e:ising-mass}, \eqref{e:sphere-density} no longer fully cancels the factor $e^{N\beta^2 q^2/2}$.
  Then
  \[
    \psi_\beta(q) = \rho(q) J_\beta(q) = \rho(q) \cdot e^{N\beta^2 q^2/2} \cdot e^{N\beta^2 /2} K_\beta(q)
  \]
  is no longer clearly log-concave, so a lower bound on $J_\beta(0)$ does not control the scale of the main contribution.
  We expect \eqref{e:critical-window-extension-goal} is still true, as BBP edge considerations similar to Remark~\ref{r:heuristic-K} suggest  
  \beq\label{e:Kbeta-critical-window}
    \log K_\beta(q) = \log K_\beta(0) - CN \lt[
      (\beta - 1 + q)_+^3 + (\beta - 1 - q)_+^3 - 2(\beta-1)_+^3 
    \rt] + \text{lower order terms}\,.
  \eeq
  If this decay rate can be proven, it overcomes the $e^{N(\beta^2-1) q^2/2}$ remaining from above, making the main contribution to \eqref{e:critical-window-extension} come from $|q| \lesssim N^{-1/3} \sqrt{\log N}$.
  This implies \eqref{e:critical-window-extension-goal}.
  However, a proof of \eqref{e:Kbeta-critical-window} would require random matrix inputs well beyond the current paper.
  After the initial posting of this paper, \eqref{e:critical-window-extension-goal} was proved by \cite{cheng2026critical} following this approach.
\end{rmk}

\subsection{Organization of the proof}
\label{ss:organization}

The rest of the paper is structured as follows.
\begin{itemize}
  \item In \S\ref{s:comparison}, we assume Proposition~\ref{p:J-of-zero}, which states that $J(0) \gtrsim N^{-1/6}$.
  Under this assumption, we prove Theorem~\ref{t:compare-cube-sphere} and Corollary~\ref{c:sphere}\ref{i:sphere-overlap-ub}.
  We also prove Proposition~\ref{p:reweight-annealed-overlap-ub}, a preliminary version of Theorem~\ref{t:overlap}\ref{i:overlap-ub}, which implies it after controlling the lower tail of $X_N$.
  These are proved by the comparison of sphere and cube overlap moments alluded to above.

  \item In \S\ref{s:rmt}, we prove Proposition~\ref{p:J-of-zero}.
  This step is the only part of the paper that uses critical-edge random matrix estimates.

  \item In \S\ref{s:variance}, we complete the proof of Theorem~\ref{t:main}.
  The main new input is Proposition~\ref{p:overlap-ub-24mt}, which provides the asymptotically sharp upper bound \eqref{e:overview-goal-overlap} on $\bbE \la R_{1,2}^2\ra$ and $\bbE \la R_{1,2}^4\ra$.
  It is proved by combining the estimates \eqref{e:overview-goal-overlap-reweight} with control of the lower tail of $X_N$ via a concentration inequality due to Chen \cite{chen2023gaussian}.
  This controls $\Var(F_N)$ through an argument from \cite{dey2026fluctuations} based on Talagrand's cavity method and an integral formula for $\Var(F_N)$ due to Chatterjee \cite{chatterjee2009disorder}.
  We also derive Theorem~\ref{t:overlap}\ref{i:overlap-lb} and Corollary~\ref{c:sphere}\ref{i:sphere-overlap-lb} as byproducts of this proof.

  \item In \S\ref{s:conclusion}, we complete the proofs of Theorem~\ref{t:overlap}\ref{i:overlap-ub} and Corollary~\ref{c:sphere}\ref{i:sphere-var}.
\end{itemize}

\section{Critical reweighting and the sphere to cube comparison}
\label{s:comparison}

From here on we let $J$ and $K$ be the functions defined in Lemma~\ref{l:reweighted-2mt}.
Recall $X_N = Z_N / Z^\sph_N$.
In this section we assume the following.
\begin{ppn}[Proved in \S\ref{s:rmt}]\label{p:J-of-zero}
  We have $J(0) \gtrsim N^{-1/6}$.
\end{ppn}
This section is devoted to the proofs of Theorem~\ref{t:compare-cube-sphere}, Corollary~\ref{c:sphere}\ref{i:sphere-overlap-ub}, and the following proposition.
\begin{ppn}\label{p:reweight-annealed-overlap-ub}
  There exists a universal constant $c>0$ such that
  \[
    \bbE\lt[
      X_N^2 \la \exp(cN^{1/2}J(0) |R_{1,2}|) \ra
    \rt] \le 2\,.
  \]
\end{ppn}
In light of Proposition~\ref{p:J-of-zero}, this implies a reweighted version of Theorem~\ref{t:overlap}\ref{i:overlap-ub}.
We keep the dependence on $J(0)$ explicit for use in \S\ref{s:variance}, see Remark~\ref{r:J-of-zero-explicit-dependence}.

\subsection{Reweighted overlap identities and log-concavity of $K$}
\label{ss:K-properties}

We first present the deferred proof of Lemma~\ref{l:reweighted-2mt} and prove Lemma~\ref{l:K-log-concave}, that $K$ is even and log-concave.

\begin{proof}[Proof of Lemma~\ref{l:reweighted-2mt}]
  For any $\bx,\by \in S_N$ with $R(\bx,\by) = q$, a Gaussian change of measure calculation shows
  \[
    J(q) 
    = e^{Nq^2/2} \cdot e^{N/2} \bbE\lt[Z^\sph_N\lt(\bW + \fr{\bx\bx^\top + \by\by^\top}{N}\rt)^{-2}\rt]\,,
  \]
  and the conclusion \eqref{e:K} follows by rotational invariance of $\bW$.
  The estimate \eqref{e:reweighted-2mt} follows from
  \[
    \bbE[X_N^2 \la g(R_{1,2}) \ra]
    = \fr{1}{4^N} \sum_{\bx,\by \in \Sigma_N} g(R(\bx,\by)) \bbE\lt[\fr{e^{H_N(\bx) + H_N(\by)}}{(Z^\sph_N)^2}\rt]
    = \bbE_q [J(q) g(q)]\,.
  \]
  The estimate \eqref{e:reweighted-sph-2mt} follows similarly as
  \[
    \bbE[\la g(R_{1,2}) \ra^\sph]
    = \iint g(R(\bx,\by)) \bbE\lt[\fr{e^{H_N(\bx) + H_N(\by)}}{(Z^\sph_N)^2}\rt] \,\de \nu^{\otimes 2}_N(\bx,\by)
    = \bbE^\sph_q [J(q) g(q)]\,. \qedhere
  \]
\end{proof}

\begin{lem}\label{l:K-log-concave}
  The function $f(q) = \log K(q)$ is even and concave, with $-N \le f''(q) \le 0$ for all $q\in [-1,1]$.
\begin{proof}
  From the definition \eqref{e:K} of $K$ it is clear that $K$, and thus $f$, is even.
  Let
  \[
    \rho_{N,\GOE}(\bW) = \fr{1}{Z_{N,\GOE}} \exp\lt(-\fr{N}{4} \|\bW\|_F^2\rt)
  \]
  be the density of $\bW \sim \GOE(N)$ in the space $\Sym_N$ of symmetric $N\times N$ real matrices.
  Since the function $\Sym_N \ni \bM \mapsto Z^\sph_N(\bM)^{-2}$ is log-concave, so is
  \[
    [-1,1] \times \Sym_N \ni (q,\bW) 
    \mapsto Z^\sph_N\lt(\bW + (1+q) \be_1\be_1^\top + (1-q) \be_2\be_2^\top\rt)^{-2} \rho_{N,\GOE}(\bW)\,.
  \]
  By the Pr\'ekopa--Leindler theorem, this implies $K$ is log-concave, and thus $f$ is concave.
  To prove the final assertion $-N \le f''(q)$, we define $\bS = \be_1\be_1^\top + \be_2\be_2^\top$ and $\bDelta = \be_1\be_1^\top - \be_2\be_2^\top$.
  Then
  \[
    K(q) = \int_{\Sym_N}
    Z^\sph_N(\bM)^{-2} \rho_{N,\GOE}\lt(\bM - \bS - q\bDelta\rt)
    \,\de \bM\,.
  \]
  Let $\zeta_q$ be the probability measure on $\Sym_N$ with density
  \[
    \zeta_q(\bM) \,\,\propto\,\, Z^\sph_N(\bM)^{-2} \rho_{N,\GOE}\lt(\bM - \bS - q\bDelta\rt)\,.
  \]
  Then
  \[
    f'(q) = (\log K)'(q) = \fr{N}{2} \int (\bM - \bS - q\bDelta, \bDelta)_F \,\de \zeta_q
    = -\fr{N}{2} q\|\bDelta\|_F^2 + \fr{N}{2} \int (\bM, \bDelta)_F \,\de \zeta_q\,.
  \]
  Differentiating again yields
  \[
    f''(q) = -\fr{N}{2} \|\bDelta\|_F^2
    + \fr{N^2}{4} \Var_{\zeta_q}[(\bM, \bDelta)_F]
    \ge -\fr{N}{2} \|\bDelta\|_F^2 = -N\,. \qedhere
  \]
\end{proof}
\end{lem}
\begin{cor}\label{c:f-deriv}
  We have $f'(q) \in [-Nq,0]$ for all $q\in [0,1]$.
\begin{proof}
  Since $f$ is even, $f'(0) = 0$.
  The result is now immediate from Lemma~\ref{l:K-log-concave}.
\end{proof}
\end{cor}

\subsection{Log-concave localization of the overlap}

This subsection presents the main mechanism for controlling tail contributions to the expectations in \eqref{e:reweighted-2mt}, \eqref{e:reweighted-sph-2mt}: for $\rho$ the spherical density defined in \eqref{e:sphere-density}, the function $\psi(q) = \rho(q) J(q)$ integrates to $1$, has central value $\psi(0) \asymp N^{1/2} J(0)$, and is essentially log-concave, so it has tails at scale $(N^{1/2} J(0))^{-1}$.
The main result of this subsection is Lemma~\ref{l:psi-tail-bound} below, which provides an exponential tail bound on $\psi$.
\begin{lem}\label{l:log-concave}
  Suppose $\mu$ is a probability measure on $\bbR$ with even and log-concave probability density $\rho$.
  Then for any $x\ge 0$, $\mu([x,+\infty)) \le \fr12 e^{-2\rho(0)x}$.
\begin{proof}
  Let $g(x) = \mu([x,+\infty))$ denote the upper tail of $\mu$.
  This is a marginal of the log-concave function $(x,y) \rightarrow \rho(x+y) \bone\{y\ge 0\}$ on $\bbR^2$, and so is log-concave by Pr\'ekopa--Leindler.
  Note that $g(0) = \fr12$ and
  \[
    (\log g)'(0) = -\fr{\rho(0)}{g(0)} = -2\rho(0)\,.
  \]
  By concavity of $\log g$ this implies
  \[
    \log g(x) \le \log g(0) - 2\rho(0) x \qquad \Rightarrow \qquad
    g(x) \le \fr12 e^{-2\rho(0) x}\,. \qedhere
  \]
\end{proof}
\end{lem}
\begin{lem}\label{l:psi}
  Recall $\rho(q)$ defined in \eqref{e:sphere-density}, and let
  \baln\label{e:psi}
    \psi(q) &= \rho(q) J(q)\,, & 
    \opsi(q) &= \psi(\max(|q|,q_\lb))\,,
  \ealn
  where $q_\lb=2N^{-1/2}$. 
  Then $\opsi$ is log-concave on $[-1,1]$, with 
  \balnn
    \opsi(0) &\gtrsim N^{1/2} J(0)\,, &
    \int_{-1}^1 \opsi(q)\,\de q &\lesssim 1\,.
  \ealnn
\begin{proof}
  Let $\phi(q) = \log \psi(q)$.
  By \eqref{e:K},
  \beq\label{e:phi}
    \phi(q) = \log \rho(0) + \fr{N}{2} + \fr{N-3}{2}\log(1-q^2)
    + \fr{Nq^2}{2} + f(q)\,.
  \eeq
  For $q\in [q_\lb,1]$, Lemma~\ref{l:K-log-concave} gives
  \[
    \phi''(q)
    = N - (N-3)\fr{1+q^2}{(1-q^2)^2} + f''(q) 
    \le N - (N-3)\fr{1+q^2}{(1-q^2)^2}
    \le 0\,.
  \]
  Also, by Corollary~\ref{c:f-deriv}, at $q = q_\lb$,
  \[
    \phi'(q)
    = q\lt(N - \fr{N-3}{1-q^2}\rt) + f'(q) 
    \le q\lt(N - \fr{N-3}{1-q^2}\rt) 
    \le 0\,.
  \]
  Since $\psi$ is even, it follows that $\opsi$ is log-concave on $[-1,1]$.
  Next, by Corollary~\ref{c:f-deriv},
  \[
    (\log J)'(q) = Nq + f'(q) \ge 0
  \]
  for $q\in [0,1]$, and therefore $J(q_\lb) \ge J(0)$.
  Since $\rho(0) \asymp N^{1/2}$ by Stirling's formula and $(1-q_\lb^2)^{(N-3)/2}\asymp 1$, we have
  \[
    \opsi(0) = \psi(q_\lb) = \rho(q_\lb) J(q_\lb) \gtrsim N^{1/2} J(0)\,.
  \]
  It remains to bound the integral of $\opsi$.
  For $0\le q\le q_\lb$, Corollary~\ref{c:f-deriv} gives
  \[
    -\fr{(N-3)q}{1-q^2}
    \le \phi'(q)
    \le q\lt(N - \fr{N-3}{1-q^2}\rt).
  \]
  Integrating and using $q_\lb=2N^{-1/2}$ shows
  \[
    |\phi(q_\lb)-\phi(q)| \le C
  \]
  for a universal constant $C$.
  Hence (after adjusting $C$) $\psi(q_\lb)\le C\psi(q)$ for all $q\in [0,q_\lb]$, and by evenness this holds for all $|q|\le q_\lb$.
  Therefore
  \[
    \int_{-1}^1 \opsi(q)\,\de q
    \le \int_{|q|\ge q_\lb} \psi(q)\,\de q + C\int_{-q_\lb}^{q_\lb} \psi(q)\,\de q
    \le C \int_{-1}^1 \psi(q)\,\de q
    = C \bbE^\sph_q[J(q)] = C\,. \qedhere
  \]
\end{proof}
\end{lem}

Lemmas~\ref{l:log-concave} and \ref{l:psi} together imply tail estimates of $\psi$, as stated in the next lemma.
\begin{lem}\label{l:psi-tail-bound}
    Let $\mu$ be the probability measure on $[-1,1]$ defined by $\mu(\de q)=\psi(q)\,\de q$ and $q_\lb=2N^{-1/2}$.
    Then, there are universal constants $c,C>0$ such that for any $q\ge q_\lb$,
    \[
      \mu[q,1]=\int_{q}^1\psi(s)\,\de s\le Ce^{-cN^{1/2}J(0)q}\,,
    \]
    and for $q\ge q_\lb+N^{-1}$,
    \[
      \psi(q)\le CN\exp\lt(-cN^{1/2}J(0)q\rt)\,.
    \]
\end{lem}
\begin{proof}
    Let $\bar{\mu}$ be the probability measure on $[-1,1]$ with density
    \[
      \bar{\mu}(\de q) = \bar{Z}^{-1} \opsi(q)\,\de q\,, \qquad\qquad
      \bar{Z} = \int_{-1}^1 \opsi(s)\,\de s\,.
    \]
    By Lemma~\ref{l:psi}, $\bar{Z} \lesssim 1$.
    Applying Lemma~\ref{l:log-concave} to $\bar{\mu}$, we obtain that for any $q\ge q_\lb$,
    \[
    \int_{q}^1\psi(s)\,\de s
    =\int_{q}^1\opsi(s)\,\de s
    \le \fr{1}{2}\exp\lt(-\fr{2\opsi(0)}{\bar{Z}}q\rt)
    \cdot \bar{Z}
    \le Ce^{-cN^{1/2}J(0)q}\,,
    \]
    where the last inequality follows from Lemma~\ref{l:psi} for some universal constants $c,C>0$.
    This establishes the first statement.

    For the second statement, first note that for any $q \ge q_\lb + N^{-1}$, we have $q - N^{-1} \ge q/2$.
    Since $\opsi$ is non-increasing on $[q_\lb,1]$, we have for any such $q$,
    \[
    \psi(q)=\opsi(q)
    \le N\int_{q-N^{-1}}^q\opsi(s)\,\de s
    \le N\int_{q-N^{-1}}^1\opsi(s)\,\de s
    \le CNe^{-cN^{1/2}J(0)(q-N^{-1})}
    \le CNe^{-cN^{1/2}J(0)q/2}\,.
    \]
    This completes the proof after adjusting $c$.
\end{proof}

We conclude this subsection by an easy bound that $J(0)\lesssim 1$.
This will be useful in later proofs in this section, though we will eventually be able to obtain the tight bound $J(0)\lesssim N^{-1/6}$ (see Corollary~\ref{c:J-of-zero-ub}).

\begin{lem}\label{l:J-of-zero-ub-crude}
    We have $J(0)\lesssim 1$. 
\end{lem}

\begin{proof}
Recall that $\psi(q)=\rho(q)J(q)$ satisfies $\int_{-1}^1 \psi(q)\,\de q = \bbE_q^\sph[J(q)] = 1$.
Let $q_\lb=2N^{-1/2}$.
In the proof of Lemma~\ref{l:psi}, we showed that $\psi(q_\lb)\le C\psi(q)$ for all $|q|\le q_\lb$.
Therefore,
\[
    1
    =
    \int_{-1}^1\psi(q)\,\de q
    \ge
    \int_{-q_\lb}^{q_\lb}\psi(q)\,\de q
    \ge
    c q_\lb\psi(q_\lb).
\]
On the other hand, Corollary~\ref{c:f-deriv} implies that $(\log J)'=Nq+f'\ge 0$, so $J$ is non-decreasing on $[0,1]$, and thus $J(q_\lb)\ge J(0)$.
Also, by Stirling's formula, $\rho(q_\lb)\asymp \sqrt N.$
Hence
\[
    \psi(q_\lb)
    =
    \rho(q_\lb)J(q_\lb)
    \ge
    c\sqrt N\,J(0).
\]
Combining the last two displays gives
\[
    1
    \ge
    c q_\lb\sqrt N\,J(0).
\]
Since $q_\lb=2N^{-1/2}$, we conclude that $J(0)\le C$.
\end{proof}

\subsection{Comparing sphere and cube overlaps}

We now proceed to prove Theorem~\ref{t:compare-cube-sphere}, Proposition~\ref{p:reweight-annealed-overlap-ub}, and Corollary~\ref{c:sphere}\ref{i:sphere-overlap-ub}.
The main idea is to compare the expectations $\bbE_q$ and $\bbE^\sph_q$ in \eqref{e:reweighted-2mt}, \eqref{e:reweighted-sph-2mt} for the cube and sphere overlap distributions.
We will divide the overlap space $[-1,1]$ into a central region near $0$ and a tail region.
In the central region, we will be able to directly compare $\bbE_q$ and $\bbE^\sph_q$, which have asymptotically equivalent densities.
In the tail region, the integrand decays exponentially by Lemma~\ref{l:psi-tail-bound}, so the contributions to both $\bbE_q$ and $\bbE^\sph_q$ are negligible.

Recall $p, \rho$ defined in \eqref{e:ising-mass}, \eqref{e:sphere-density}, and $\cQ_N = \{-1,-1+\fr{2}{N},\ldots,1\}$.
We start by the following density comparison lemma.

\begin{lem}\label{l:cube-sphere-equivalence}
    For any $q \in \cQ_N$ such that $|q|=o(N^{-1/4})$ and $\tilde q\in [q-N^{-1},q+N^{-1}]$, we have
    \[
    p(q)e^{Nq^2/2}=\rho(\tilde q)e^{N\tilde q^2/2}\cdot \fr{2}{N}\cdot \lt(1+O\lt(N^{-1}+q^2+Nq^4\rt)\rt)\,.
    \]
    The implicit constant in the $O(\cdot)$ is uniform in $q,\tilde q$.
\end{lem}
\begin{proof}
    For any $q \in \cQ_N$ with $|q|=o(N^{-1/4})$, Stirling's formula gives
    \[
    p(q)e^{Nq^2/2}=2^{-N}\binom{N}{\fr{1+q}{2}N}e^{Nq^2/2}=\sqrt{\fr{2}{\pi N}}\lt(1+O\lt(N^{-1}+q^2+Nq^4\rt)\rt)\,,
    \]
    and for any $\tilde q\in [q-N^{-1},q+N^{-1}]$,
    \[
    \rho(\tilde q)e^{N\tilde q^2/2}
    =\fr{\Gamma(N/2)}{\sqrt\pi\Gamma((N-1)/2)}(1-\tilde q^2)^{(N-3)/2}e^{N\tilde q^2/2}
    =\sqrt{\fr{N}{2\pi}}\lt(1+O(N^{-1}+\tilde q^2+N\tilde q^4)\rt)\,.
    \]
The desired result then follows. 
\end{proof}

In what follows, we set $q_\lb=2N^{-1/2}$ and $a=N^{1/2}J(0)$.
Moreover, we fix a large constant $K>0$ and let $q_*=Ka^{-1}$.
By Lemma~\ref{l:J-of-zero-ub-crude} we can pick $K$ large enough such that $q_*\ge q_\lb$.

\begin{proof}[Proof of Theorem~\ref{t:compare-cube-sphere}]
Recall from \eqref{e:overview-goal-compare-step1} and Lemma~\ref{l:reweighted-2mt} that
\[
    \bbE[(X_N-1)^2]
    =
    \bbE_q[J(q)]-\bbE_q^\sph[J(q)]=\sum_{q\in \cQ_N}p(q)J(q)-\int_{-1}^1\rho(q)J(q)\,\de q.
\]
Let $q_\ub=L a^{-1} \log N$, where $L$ is a sufficiently large universal constant.
Let $A=\cQ_N\cap [-q_\ub,q_\ub]$ and $I_A=A+[-N^{-1},N^{-1}]$.
We decompose
\[
\begin{aligned}
\bbE[(X_N-1)^2]
&\le
\left|
    \sum_{q\in A}p(q)J(q)
    -
    \int_{I_A}\rho(\tilde q)J(\tilde q)\,\de \tilde q
\right|  
+
\sum_{q\in\cQ_N\setminus A}p(q)J(q)
+
\int_{[-1,1]\setminus I_A}\rho(q)J(q)\,\de q .
\end{aligned}
\]
We first bound the spherical tail.
Recall that $\psi(q)=\rho(q)J(q)$.
Note that $q_\ub - N^{-1} \ge q_\lb$, while Lemma~\ref{l:J-of-zero-ub-crude} gives
\[
  a(q_\ub - N^{-1}) = L \log N - O(N^{-1/2})\,.
\]
Choosing $L$ large then ensures that, by Lemma~\ref{l:psi-tail-bound},
\balnn
    \int_{[-1,1]\setminus I_A}\rho(q)J(q)\,\de q
    &\le
    2\int_{q\ge q_\ub - N^{-1}}\psi(q)\,\de q
    \le
    C\exp(-ca(q_\ub - N^{-1}))
    \lesssim N^{-10}\,.
\ealnn
Next we bound the cube tail.
We use the factorization
\[
    p(q)J(q)
    =
    p(q)e^{Nq^2/2}\cdot e^{N/2}K(q).
\]
Uniformly over $q\in\cQ_N$,
\[
    p(q)e^{Nq^2/2}\le \fr{C}{\sqrt N}.
\]
Since $f=\log K$ is even and concave by Lemma~\ref{l:K-log-concave}, $K$ is non-increasing on $[0,1]$.
Hence for $|q|\ge q_\ub$,
\[
    e^{N/2}K(q)\le e^{N/2}K(q_\ub).
\]
By the pointwise estimate in Lemma~\ref{l:psi-tail-bound},
\[
    \psi(q_\ub)
    \le
    CN\exp(-caq_\ub).
\]
By Proposition~\ref{p:J-of-zero},
\[
  q_\ub \lesssim (N^{1/2} J(0))^{-1} \log N \lesssim N^{-1/3} \log N \ll N^{-1/4}\,.
\]
So, Stirling's formula gives
\[
  \psi(q_\ub)
  = \rho(q_\ub)e^{Nq_\ub^2/2}e^{N/2}K(q_\ub)
  \asymp \sqrt{N} e^{N/2}K(q_\ub)\,.
\]
Therefore,
\[
    e^{N/2}K(q_\ub)
    \le C\sqrt N\exp(-caq_\ub)\le 
    CN^{-cL + \fr12}\,.
\]
Consequently, for sufficiently large $L$,
\[
    \sum_{q\in\cQ_N\setminus A}p(q)J(q)
    \le
    \sum_{q\in\cQ_N\setminus A}
    \fr{C}{\sqrt N}
    \cdot
    CN^{-cL + \fr12}
    \le
    CN^{-cL+1} \lesssim N^{-10}\,.
\]
Thus, both the cube and spherical tail have negligible contribution.

It remains to compare the central sum with the central integral.
For $q\in A$ and $|\tilde q-q|\le N^{-1}$, Lemma~\ref{l:cube-sphere-equivalence} gives
\[
    p(q)e^{Nq^2/2}
    =
    \rho(\tilde q)e^{N\tilde q^2/2}\cdot
    \fr{2}{N}
    \cdot\lt(1+O(N^{-1}+q^2+Nq^4)\rt).
\]
Also, by Corollary~\ref{c:f-deriv}, for such $q,\tilde q$, (recall $f = \log K$)
\[
    |f(q)-f(\tilde q)|
    \le
    N(|q|+N^{-1})|\tilde q-q|
    \le
    |q|+N^{-1}.
\]
Thus
\[
    e^{N/2}K(q)
    =
    e^{N/2}K(\tilde q)
    \lt(1+O(|q|+N^{-1})\rt),
\]
and hence
\[
    p(q)J(q)
    =
    \lt(1+O(N^{-1}+q^2+Nq^4+|q|)\rt)
    \int_{q-N^{-1}}^{q+N^{-1}}
        \rho(\tilde q)J(\tilde q)\,\de \tilde q .
\]
Summing over $q\in A$, we get
\begin{align*}
\Bigg|
    \sum_{q\in A}p(q)J(q)
    -
    \int_{I_A}\rho(q)J(q)\,\de q
\Bigg|
\le
    C
    \int_{I_A}(N^{-1}+q^2+Nq^4+|q|)\rho(q)J(q)\,\de q .
\end{align*}
We claim the right-hand side integral is $O(N^{-1/3})$.
To see this, recall that $\mu$ is the probability measure on $[-1,1]$ defined by $\mu(\de q)=\psi(q)\,\de q=\rho(q)J(q)\,\de q$.
By evenness of $\mu$ and integration by parts,
\[
  \int_{I_A}(N^{-1}+q^2+Nq^4+|q|)\rho(q)J(q)\,\de q  \le N^{-1}+2\int_0^1(1+2q+4Nq^3)\mu[q,1]\,\de q.
\]
Recall $a=N^{1/2}J(0)$ and $q_*=Ka^{-1}\ge q_\lb$.
For $0\le q\le q_*$, we have $(1+2q+4Nq^3)\mu[q,1]\le 1+2q_*+4Nq_*^3$, and thus
\[
\int_0^{q_*}(1+2q+4Nq^3)\mu[q,1]\,\de q\le C(a^{-1}+Na^{-4})\,.
\]
For $q\ge q_*$, using the tail estimate of $\mu$ in Lemma~\ref{l:psi-tail-bound}, we have
\[
    \int_{q_*}^1(1+2q+4Nq^3)\mu[q,1]\,\de q
    \le C\int_{Ka^{-1}}^1(1+Nq^3)e^{-caq}\,\de q
    = C\int_{K}^a(a^{-1}+Na^{-4}s^3)e^{-cs}\,\de s\,,
\]
where the last identity follows by the change of variables $s=aq$. 
It is straightforward to show that the last integral is upper bounded by $C(a^{-1}+Na^{-4})$.
Combining these estimates together, the claim follows.

This shows the central contribution is at most
\[
C(a^{-1}+Na^{-4})\lesssim \fr{1}{N^{1/2}J(0)}+\fr{1}{NJ(0)^4}
\lesssim N^{-1/3}\,,
\]
where in the last estimate we used Proposition~\ref{p:J-of-zero}.
As we showed above the tail contributions are lower order, the result follows.
\end{proof}

To prove Proposition~\ref{p:reweight-annealed-overlap-ub} and Corollary~\ref{c:sphere}\ref{i:sphere-overlap-ub}, we need the following lemma. 

\begin{lem}\label{lem:spherical-exp-moment}
There exists a universal constant $c_0>0$ such that, for all $0<c\le c_0$,
\[
    \int_{-1}^1 \rho(q)J(q)
    \exp(ca|q|)\,\de q
    \le 1.1 .
\]
\end{lem}

\begin{proof}
Recall that $a = N^{1/2} J(0)$ and $q_*=Ka^{-1}$.
We will choose $K$ sufficiently large and $c$ sufficiently small depending on $K$.
We split the exponential moment at $q_*$.
On the central region,
\[
    \int_{|q|\le q_*}\psi(q)e^{ca|q|}\,\de q
    \le
    e^{cK}\int_{|q|\le q_*}\psi(q)\,\de q .
\]
For the tail, by the evenness and integration by parts,
\[
\begin{aligned}
    \int_{|q|\ge q_*}\psi(q)e^{ca|q|}\,\de q
    &\le
    2e^{caq_*}\mu[q_*,1]
    +
    2ca\int_{q_*}^1 e^{cat}\mu[t,1]\,\de t
\end{aligned}
\]
By Lemma~\ref{l:psi-tail-bound}, for all $t\ge q_*$, $\mu[t,1]\le C e^{-c_1at}$ for some universal constants $c_1,C>0$.
Therefore,
\[
\begin{aligned}
    \int_{|q|\ge q_*}\psi(q)e^{ca|q|}\,\de q
    &\le
 Ce^{-(c_1-c)K}
    +
    Cca\int_{q_*}^{\infty} e^{-(c_1-c)at}\,\de t,
\end{aligned}
\]
which is bounded by $C e^{-(c_1-c)K}$ for some enlarged $C$. 
Now choose $K$ sufficiently large so that $C e^{-c_1K/2}\le 0.01$, and then choose $c_0>0$ sufficiently small so that $c_0\le c_1/2$ and $e^{c_0K}\le 1.01$.
For every $0<c\le c_0$, the preceding estimates give
\[
    \int_{-1}^1\psi(q)e^{ca|q|}\,\de q
    \le
    e^{cK}\int_{|q|\le q_*}\psi(q)\,\de q
    +
    0.01 
    \le
    1.01\int_{|q|\le q_*}\psi(q)\,\de q
    +
    0.01
    \le 1.02 .
\]
This implies the desired bound.
\end{proof}

We now provide the proof of Proposition~\ref{p:reweight-annealed-overlap-ub} and Corollary~\ref{c:sphere}\ref{i:sphere-overlap-ub}. 

\begin{proof}[Proof of Proposition~\ref{p:reweight-annealed-overlap-ub}]
As before, we let $L$ be a sufficiently large universal constant and
\[
    q_\ub=La^{-1} \log N,
    \qquad
    A=\cQ_N\cap[-q_\ub,q_\ub],\qquad I_A=A+[-N^{-1},N^{-1}].
\]
By \eqref{e:reweighted-2mt} applied to $g(q)=\exp(ca|q|)$, 
\[
\begin{aligned}
    \bbE\lt[
        X_N^2
        \la
            \exp(c a |R_{1,2}|)
        \ra
    \rt]
    &=
    \bbE_q\lt[
        J(q)\exp(c a |q|)
    \rt]
    &=
    \sum_{q\in\cQ_N}
        p(q)J(q)\exp(c a |q|).
\end{aligned}
\]
We compare this cube sum with the corresponding spherical integral
\[
    \int_{-1}^1 \rho(q)J(q)\exp(c a |q|)\,\de q.
\]

By Lemma~\ref{lem:spherical-exp-moment}, after choosing $c>0$ sufficiently small,
\[
    \int_{-1}^1 \rho(q)J(q)\exp(c a |q|)\,\de q
    \le 1.1 .
\]
It remains to bound the error from replacing the cube sum by the spherical integral.
We first consider the central region.

For $q\in A$ and $|\tilde q-q|\le N^{-1}$, Lemma~\ref{l:cube-sphere-equivalence} gives
\[
    p(q)e^{Nq^2/2}
    =
    \rho(\tilde q)e^{N\tilde q^2/2}
    \cdot \fr{2+o(1)}{N}.
\]
Moreover, by Corollary~\ref{c:f-deriv},
\[
    |f(q)-f(\tilde q)|
    \le
    N(q_\ub+N^{-1})|\tilde q-q|
    \le q_\ub+N^{-1}=o(1),
\]
and also by Lemma~\ref{l:J-of-zero-ub-crude},
\[
    ca\bigl||q|-|\tilde q|\bigr|
    \le
    caN^{-1}
    =o(1).
\]
Therefore,
\[
    J(q)e^{ca|q|}
    =
    \lt(1+o(1)\rt)\cdot J(\tilde q)e^{ca|\tilde q|}
    .
\]
Combining the above estimates, for every $q\in A$,
\[
\begin{aligned}
    p(q)J(q)e^{ca|q|}
    =
    \lt(1+o(1)\rt)
    \int_{q-N^{-1}}^{q+N^{-1}}
        \rho(\tilde q)J(\tilde q)e^{ca|\tilde q|}\,\de \tilde q .
\end{aligned}
\]
Summing over $q\in A$ and using Lemma~\ref{lem:spherical-exp-moment}, we obtain
\[
\begin{aligned}
    \sum_{q\in A}p(q)J(q)e^{ca|q|}
    &\le
    \lt(1+o(1)\rt)
    \int_{I_A}\rho(q)J(q)e^{ca|q|}\,\de q
    \le
    1.1+o(1).
\end{aligned}
\]
Thus the central contribution is at most $1.2$, for all sufficiently large $N$.

It remains to bound the cube tail.
We use two different estimates, according to the size of $q$.
By evenness it suffices to consider $q\ge q_\ub$.

First, by the binomial and spherical density formula, Stirling's formula and Taylor expansion give
\[
p(q)
=
(1+o(1)) \sqrt{\fr{2}{\pi N}}
\exp\lt\{
-\fr N2 q^2-\fr N{12}q^4+O(Nq^6)
\rt\},
\]
while
\[
\rho(q)
=
(1+o(1)) \sqrt{\fr{N}{2\pi}}
\exp\lt\{
-\fr N2 q^2-\fr N4 q^4+O(Nq^6)
\rt\}.
\]
Consequently, for $q\in \cQ_N$ such that $Nq^6\le 1$, it follows that
\[
\fr{p(q)}{\rho(q)}
\le
\fr{C}{N}
\exp\lt\{\fr16Nq^4\rt\}.
\]
Hence, using the pointwise estimate from Lemma~\ref{l:psi-tail-bound}, for $q_\ub\le q\le N^{-1/6}$,
\[
\begin{aligned}
    p(q)J(q)e^{ca q}
    &\le
    \fr{C}{N}
    \exp\lt\{\fr{1}{6}Nq^4\rt\}
    \rho(q)J(q)e^{ca q}  \le
    C
    \exp\lt\{\fr{1}{6}Nq^4-c_1aq+caq\rt\}.
\end{aligned}
\]
In particular, by choosing $c>0$ small depending on $c_1$ and $L$ large depending on $c,c_1$, we have that whenever $q\ge q_\ub$ and $Nq^4\le 12caq\le \fr{c_1}{4}aq$ (which automatically implies $q\le N^{-1/6}$ by Lemma~\ref{l:J-of-zero-ub-crude}),
\[
    p(q)J(q)e^{caq}
    \le
    C e^{-c_1aq/4}
    \le
    C e^{-c_1aq_\ub/4}
    = CN^{-Lc_1/4}
    \lesssim N^{-10}\,.
\]
On the other hand, we also have the crude bound for any $q\in \cQ_N$,
\[
    p(q)e^{Nq^2/2}
    \le
    C\exp\lt\{-\fr{1}{12}Nq^4\rt\}.
\]
Indeed, by Stirling's formula, $p(q) \le C\exp(-NI(q))$, where $I(q)$ is the entropy function.
Since $I(q)-q^2/2\ge q^4/12$ for all $|q|\le 1$, the result follows.
Meanwhile, since $f=\log K$ is even and concave, $K$ is non-increasing on $[0,1]$.
Therefore, for $q\ge q_\ub$,
\[
    e^{N/2}K(q)\le e^{N/2}K(q_\ub)\le 
    CN^{-cL + \fr12}
\]
as in the proof of Theorem~\ref{t:compare-cube-sphere}.
Thus for any $q\ge q_\ub$,
\[
    p(q)J(q)e^{caq}
    \le
    C\exp\lt\{-\fr{1}{12}Nq^4+caq\rt\}\cdot N^{-cL + \fr12}\,.
\]
This gives that whenever $Nq^4\ge 12caq$, for $L$ sufficiently large,
\[
    p(q)J(q)e^{caq}
    \le
    CN^{-cL + \fr12}
    \lesssim N^{-10}\,.
\]
Combining the above two bounds, we conclude by choosing $c>0$ small enough that
\[
    \sup_{q\in\cQ_N\setminus A}
    p(q)J(q)e^{ca|q|}
    \ll
    N^{-1}.
\]
Since $|\cQ_N|\le N+1$, we obtain
\[
    \sum_{q\in\cQ_N\setminus A}
        p(q)J(q)e^{ca|q|}=o(N\cdot N^{-1})=o(1)\,.
\]

Combining the central estimate and the tail estimate, and decreasing $c>0$ if necessary, we get
\[
    \bbE\lt[
        X_N^2
        \la
            \exp(ca|R_{1,2}|)
        \ra
    \rt]
    \le 2 .
\]
This proves the proposition.
\end{proof}

\begin{proof}[Proof of Corollary~\ref{c:sphere}\ref{i:sphere-overlap-ub}]
By \eqref{e:reweighted-sph-2mt} applied to $g(q)=\exp(ca|q|)$, we have
\[
    \bbE\lt[
        \la
            \exp(ca|R_{1,2}|)
        \ra^\sph
    \rt]
    =
    \int_{-1}^1
        \rho(q)J(q)
        \exp(ca|q|)\,\de q .
\]
By Lemma~\ref{lem:spherical-exp-moment}, there exists $c_0$ such that for any $c'\le c_0$, 
\[
    \bbE\lt[
        \la
            \exp(c'a|R_{1,2}|)
        \ra^\sph
    \rt]
    \le 1.1 .
\]
On the other hand, Proposition~\ref{p:J-of-zero} gives $J(0)\ge c_1N^{-1/6}$, and hence $a=N^{1/2}J(0)\ge c_1N^{1/3}$.
Therefore, for small enough $c$ we may pick $c'\le c_0$ such that $c'a=cN^{1/3}$, and thus
\[
    \bbE\lt[
        \la
            \exp(cN^{1/3}|R_{1,2}|)
        \ra^\sph
    \rt]
    \le 1.1
    \le 2 .
\]
This proves Corollary~\ref{c:sphere}\ref{i:sphere-overlap-ub}.
\end{proof}

\section{The spherical two-replica ratio at criticality}
\label{s:rmt}

In this section, we will prove Proposition~\ref{p:J-of-zero}.
Define the Stiefel manifold
\[
  S_{N,2} = \lt\{(\bx,\by) \in S_N \times S_N: R(\bx,\by) = 0\rt\}\,,
\]
and let $\nu_{N,2}$ denote the normalized Haar measure on $S_{N,2}$.
Then define
\[
  Z^\sph_{N,2} = \int_{S_{N,2}} e^{H_N(\bx) + H_N(\by)} \,\de \nu_{N,2}(\bx,\by)\,.
\]
By rotational invariance it is clear that
\beq\label{e:J-of-zero-rotational-invariance}
  J(0) = \bbE\lt[\fr{Z^\sph_{N,2}}{(Z^\sph_N)^2} \rt]\,.
\eeq
Thus it suffices to show that this ratio is at least of order $N^{-1/6}$ with probability bounded away from zero.
More precisely, Proposition~\ref{p:J-of-zero} follows immediately from the following.
\begin{ppn}\label{ppn:positive-probability-lower-bound}
  There exists a universal constant $c>0$ such that
  \[
    \bbP\lt(\fr{Z^\sph_{N,2}}{(Z^\sph_N)^2} \ge cN^{-1/6}\rt) \ge c\,.
  \]
\end{ppn}
To prove Proposition~\ref{ppn:positive-probability-lower-bound}, we analyze the numerator and denominator separately.
By rotational invariance, both $Z_N^{\sph}(\bW)$ and $Z_{N,2}^{\sph}(\bW)$ depend on $\bW$ only through its spectrum
\[
\bLambda=(\lambda_1,\dots,\lambda_N).
\]
Given $\bLambda$, define
\begin{equation}\label{eq:G_Lambda}
G_\bLambda(z)
=
z-\fr1N\sum_{k=1}^N\log(z-\lambda_k),
\end{equation}
where $\log(\cdot)$ denotes the principal branch on $\bbC\setminus(-\infty,0]$.
Let $\gamma=\gamma(\bLambda)$ be the unique solution of $G_\bLambda'(\gamma)=0$ in $(\lambda_1,\infty)$.
We will show that, with probability bounded away from zero over $\bW\sim\GOE(N)$,
\[
Z_N^\sph(\bW)
\lesssim
e^{\fr N2(G_\bLambda(\gamma)-1)}\,N^{-1/6},
\qquad
Z_{N,2}^\sph(\bW)
\gtrsim
e^{N(G_\bLambda(\gamma)-1)}\,N^{-1/2}.
\]
Taking the ratio immediately yields Proposition~\ref{ppn:positive-probability-lower-bound}.

The starting point is to obtain contour integral representations for $Z_N^\sph$ and $Z_{N,2}^\sph$.
These formulas express both partition functions entirely in terms of the spectrum $\bLambda$ and reduce the problem to a saddle-point analysis near the critical edge of the GOE spectrum.
We derive these representations in \S\ref{subsec:integral-representation}.
In \S\ref{subsec:estimate-ZNsph}, we analyze the resulting contour integrals on a typical event for $\bW\sim\GOE(N)$ and establish the above estimates, thereby proving Proposition~\ref{ppn:positive-probability-lower-bound} and hence Proposition~\ref{p:J-of-zero}.

\subsection{One- and two-replica contour representations}\label{subsec:integral-representation}
In this subsection, we derive contour integral representations for $Z_N^{\sph}$ and $Z_{N,2}^{\sph}$.
These formulas express both partition functions entirely in terms of the spectrum of $\bW$, and will serve as the starting point for the critical-edge analysis in the next subsection.
Throughout this subsection, we fix a matrix $\bW$ and write $\bLambda=(\lambda_1,\dots,\lambda_N)$ for its spectrum, ordered as $\lambda_1\ge \cdots \ge \lambda_N$.

We begin with the one-replica partition function $Z_N^{\sph}$.
Recall the function $G_\bLambda$ defined in \eqref{eq:G_Lambda}, and let $\gamma=\gamma(\bLambda)$ denote the unique critical point of $G_\bLambda$ in $(\lambda_1,\infty)$.

\begin{lem}\label{lem:contour-integral-1}
It holds that
    \[
Z_N^\sph(\bW)=\fr{C_N}{2\pi i}\int_{\gamma-i\infty}^{\gamma+i\infty}\exp\left(\fr{N}{2}G_\bLambda(z)\right)\de z\,,
    \]
where $C_N$ is the nonrandom number given by
\[
C_N=\fr{\Gamma(N/2)}{(N/2)^{N/2-1}}=(1+o(1)) \sqrt{\pi N}e^{-N/2}\,.
\]
\end{lem}
\begin{proof}
    The result has appeared in \cite[Lemma 1.3]{baik2016fluctuations}, but we will still give a full proof for completeness and to motivate the analysis for the more complicated term $Z_{N,2}^\sph$. 

    Since $\bW$ has spectrum $\bLambda$, by change of variables we get
    \[
        Z_N^{\sph}(\bW)
        =
        \int_{S_N}
        \exp\left\{
            \fr12\sum_{k=1}^N \lambda_k x_k^2
        \right\}
        \de\nu_N(\bx).
    \]
    For $t>0$, define
    \[
        M_\bLambda(t)
        =
        \int_{S_N}
        \exp\left\{
            \fr t2\sum_{k=1}^N \lambda_k x_k^2
        \right\}
        \de\nu_N(\bx).
    \]
    Thus $Z_N^{\sph}(\bW)=M_\bLambda(1)$.

    We compute the Laplace transform of $t^{N/2-1}M_\bLambda(t)$.
    For any $z>\lambda_1$,
    \[
        L_\bLambda(z)
        :=
        \int_0^\infty
        e^{-Nzt/2}t^{N/2-1}M_\bLambda(t)\,\de t .
    \]
    Using polar coordinates $\by=r\bomega$ in $\bbR^N$, and writing $\bx=\sqrt N\,\bomega\in S_N$, we have
    \[
        \de\by
        =
        |\bbS^{N-1}|\,r^{N-1}\de r\,\de\nu_N(\bx)
        =
        \fr{|\bbS^{N-1}|N^{N/2}}{2}
        t^{N/2-1}\de t\,\de\nu_N(\bx),
        \qquad t=\fr{r^2}{N}.
    \]
    Therefore
    \[
        \int_{\bbR^N}
        \exp\left\{
            -\fr12\sum_{k=1}^N (z-\lambda_k)y_k^2
        \right\}\,\de\by=
        \fr{|\bbS^{N-1}|N^{N/2}}{2}
        \int_0^\infty
        e^{-Nzt/2}t^{N/2-1}M_\bLambda(t)\,\de t
        =
        \fr{|\bbS^{N-1}|N^{N/2}}{2}
        L_\bLambda(z).
    \]
    On the other hand, the same Gaussian integral factorizes as
    \[
        \int_{\bbR^N}
        \exp\left\{
            -\fr12\sum_{k=1}^N (z-\lambda_k)y_k^2
        \right\}\de\by
        =
        (2\pi)^{N/2}
        \prod_{k=1}^N (z-\lambda_k)^{-1/2}.
    \]
    Since
    \[
        |\bbS^{N-1}|=\fr{2\pi^{N/2}}{\Gamma(N/2)},
    \]
    we obtain
    \[
        L_\bLambda(z)
        =
        \Gamma(N/2)
        \left(\fr{2}{N}\right)^{N/2}
        \prod_{k=1}^N (z-\lambda_k)^{-1/2}.
    \]
    By the inverse Laplace transform formula (i.e. the Bromwich inversion formula \cite{widder1946laplace}), for any $\gamma>\lambda_1$ (and we choose $\gamma$ as the saddle point $\gamma = \gamma(\bLambda)$ for later convenience),
    \[
        M_\bLambda(1)
        =
        \fr{N}{2}\fr{1}{2\pi i}
        \int_{\gamma-i\infty}^{\gamma+i\infty}
        e^{Nz/2}L_\bLambda(z)\,\de z .
    \]
    Substituting the expression for $L_\bLambda(z)$ gives
    \[
        Z_N^{\sph}(\bW)
        =
        \fr{\Gamma(N/2)}{(N/2)^{N/2-1}}
        \fr{1}{2\pi i}
        \int_{\gamma-i\infty}^{\gamma+i\infty}
        \exp\left\{
            \fr N2 z
            -\fr12\sum_{k=1}^N \log(z-\lambda_k)
        \right\}
        \de z .
    \]
    Since the contour lies to the right of the spectrum,
    \[
        z-\fr1N\sum_{k=1}^N\log(z-\lambda_k)
        =
        G_\bLambda(z).
    \]
    Hence
    \[
        Z_N^{\sph}(\bW)
        =
        \fr{C_N}{2\pi i}
        \int_{\gamma-i\infty}^{\gamma+i\infty}
        \exp\left\{
            \fr N2 G_\bLambda(z)
        \right\}
        \de z,
        \qquad
        C_N
        =
        \fr{\Gamma(N/2)}
             {(N/2)^{N/2-1}} .
    \]
    Finally, Stirling's formula gives
    \[
        C_N
        =
        (1+o(1))\sqrt{\pi N}\,e^{-N/2},
    \]
    which completes the proof. 
\end{proof}

Now we turn to the more complicated $Z_{N,2}^\sph$.
Since $Z_{N,2}^{\sph}$ involves two orthogonal replicas, the natural analogue of the scalar Laplace transform is a matrix-valued Laplace transform on the cone $\Sym_2^+$ of positive definite $2\times 2$ symmetric matrices.
For $Z\in \Sym_2(\bbC)$, write
    \[
        Z=
        \begin{pmatrix}
            z_1 & w \\
            w & z_2
        \end{pmatrix}.
    \]
    Define (again, $\log(\cdot)$ denotes the principal branch on $\bbC\setminus(-\infty,0]$)
    \[
        G_{\bLambda,2}(z_1,z_2,w)
        =
        z_1+z_2
        -
        \fr1N
        \sum_{k=1}^N
        \log\bigl((z_1-\lambda_k)(z_2-\lambda_k)-w^2\bigr).
    \]

\begin{lem}\label{lem:spherical-two-contour}
It holds that
    \[
        Z_{N,2}^{\sph}(\bW)
        =
        \fr{C_{N,2}}{(2\pi i)^3}
        \int_{(\gamma,\gamma,0)+(i\bbR)^3}
        \exp\left\{
            \fr N2
            G_{\bLambda,2}(z_1,z_2,w)
        \right\}
        \de z_1\,\de z_2\,\de w ,
    \]
where
\[
    C_{N,2}=\fr{2\sqrt{\pi}\Gamma(N/2)\Gamma((N-1)/2)}{(N/2)^{N-3}} =(1+o(1))2^{1/2}\pi^{3/2}
    N^{3/2} e^{-N}\,.
\]
\end{lem}

\begin{proof}
    By slight abuse of notation, let $\bLambda = \operatorname{diag}(\lambda_1,\dots,\lambda_N)$.
    By rotational invariance we may assume $\bW = \bLambda$.
    Write
    \[
        \bx=\sqrt N \bu,\qquad \by=\sqrt N \bv,
    \]
    where
    \[
        \bQ=(\bu,\bv)\in V_{N,2}
        :=
        \{\bQ\in \bbR^{N\times 2}: \bQ^\top \bQ=I_2\}.
    \]
    Then
    \[
        H_N(\bx)+H_N(\by)
        =
        \fr N2 \Tr(\bQ^\top \bLambda \bQ),
    \]
    and therefore
    \[
        Z_{N,2}^{\sph}(\bW)
        =
        \int_{V_{N,2}}
        \exp\left\{
            \fr N2 \Tr(\bQ^\top \bLambda \bQ)
        \right\}
        \de\nu_{N,2}(\bQ).
    \]
    For $S\in \Sym_2^+$, define
    \[
        M_\bLambda(S)
        =
        \int_{V_{N,2}}
        \exp\left\{
            \fr N2
            \Tr\bigl(S^{1/2}\bQ^\top \bLambda \bQ S^{1/2}\bigr)
        \right\}
        \de\nu_{N,2}(\bQ).
    \]
    Thus $Z_{N,2}^\sph(\bW)=M_\bLambda(I_2)$.
    
    We compute the Laplace transform of $\det(S)^{(N-3)/2}M_\bLambda(S)$.
    For $Z\in \Sym_2(\bbR)$ satisfying $Z-\lambda_1 I_2>0$, set
    \[
        L_\bLambda(Z)
        =
        \int_{\Sym_2^+}
        e^{-\fr N2\Tr(ZS)}
        \det(S)^{(N-3)/2}
        M_\bLambda(S)\,\de S .
    \]
    For $\bX\in \bbR^{N\times 2}$, use the matrix polar decomposition
    \[
        \bX=\bQ S^{1/2},\qquad \bQ\in V_{N,2},\qquad S=\bX^\top \bX\in \Sym_2^+ .
    \]
    The matrix polar-coordinate formula gives
    \[
        \de \bX
        =
        c_{N,2}\det(S)^{(N-3)/2}\,\de S\,\de\nu_{N,2}(\bQ),
    \]
    where \[
    c_{N,2}
    =
    \fr{\pi^N}{\Gamma_2(N/2)}=
    \fr{\pi^N}
    {\sqrt{\pi}\,\Gamma(N/2)\Gamma((N-1)/2)} .
\]
Hence
    \[
        L_\bLambda(Z)
        =
        c_{N,2}^{-1}
        \int_{\bbR^{N\times 2}}
        \exp\left\{
            -\fr N2\Tr(Z\bX^\top \bX)
            +
            \fr N2\Tr(\bX^\top \bLambda \bX)
        \right\}
        \,\de \bX .
    \]
    Writing $x_k\in \bbR^2$ for the $k$-th row of $\bX$, the integral factorizes:
    \[
        L_\bLambda(Z)
        =
        c_{N,2}^{-1}
        \prod_{k=1}^N
        \int_{\bbR^2}
        \exp\left\{
            -\fr N2 x_k^\top (Z-\lambda_k I_2)x_k
        \right\}
        \de x_k .
    \]
    Therefore
    \[
        L_\bLambda(Z)
        =
        c_{N,2}^{-1}
        \left(\fr{2\pi}{N}\right)^N
        \prod_{k=1}^N
        \det(Z-\lambda_k I_2)^{-1/2}.
    \]
By the matrix inversion formula for the Laplace transform on $\Sym_2$ (see, e.g., \cite[proof of Proposition XV.2.2]{FK94}, where the inversion formula is used), evaluating at $S=I_2$ gives
    \begin{equation*}\label{e:2-replica-contour}
        M_\bLambda(I_2)
        =
        \fr{\tilde{C}_{N,2}}{(2\pi i)^3}
        \int_{\Gamma}
        e^{\fr N2\Tr Z}
        \prod_{k=1}^N
        \det(Z-\lambda_k I_2)^{-1/2}
        \de Z ,
    \end{equation*}
    where $\de Z$ denotes the Lebesgue measure on $\Sym_2(\bbC)$ dual to the standard coordinate measure $\de S=\de s_{11}\de s_{22}\de s_{12}$ under the trace pairing.
Moreover, the constant $\tilde{C}_{N,2}$ is explicitly given by
\[
    \tilde{C}_{N,2}
    =\left(\fr N 2\right)^3\left(\fr{2\pi}{N}\right)^N c_{N,2}^{-1}=
    \fr{\sqrt{\pi}\,
    \Gamma(N/2)\Gamma((N-1)/2)}
    {(N/2)^{N-3}},
\]
and $\Gamma=\Gamma_0+i\Sym_2(\bbR)$ for any $\Gamma_0-\lambda_1I_2>0$ (and we take $\Gamma_0=\gamma I_2$ for later convenience).
Alternatively, from a more elementary perspective, since $\Sym_2\cong \mathbb{R}^3$, one may also derive the above contour representation by applying the classic Bromwich inversion formula \cite{widder1946laplace} successively in the three coordinates.

Writing
\[
Z=\begin{pmatrix}z_1&w\\ w&z_2\end{pmatrix},
\qquad
\Tr(ZS)=z_1s_{11}+z_2s_{22}+2ws_{12},
\]
we have
\[
\de Z = \de z_1\,\de z_2\,\de (2w)=2\,\de z_1\,\de z_2\,\de w,
\]
and thus integration over $\Gamma$ of $\de Z$ is equivalent to integration over $(\gamma,\gamma,0)+(i\bbR)^3$ of $2\de z_1\de z_2 \de w$.
Moreover,
    \[
        \Tr Z=z_1+z_2,\quad 
        \det(Z-\lambda_k I_2)
        =
        (z_1-\lambda_k)(z_2-\lambda_k)-w^2.
    \]
    Hence
    \[
        e^{\fr N2\Tr Z}
        \prod_{k=1}^N
        \det(Z-\lambda_k I_2)^{-1/2}
        =
        \exp\left\{
            \fr N2
            G_{\bLambda,2}(z_1,z_2,w)
        \right\}.
    \]
This yields that
    \[
        Z_{N,2}^{\sph}(\bW)
        =
        M_\bLambda(I_2)
        =
        \fr{C_{N,2}}{(2\pi i)^3}
        \int_{(\gamma,\gamma,0)+(i\bbR)^3}
        \exp\left\{
            \fr N2
            G_{\bLambda,2}(z_1,z_2,w)
        \right\}
        \de z_1\,\de z_2\,\de w ,
    \]
where $C_{N,2}=2\tilde{C}_{N,2}$ as desired.
The asymptotics of $C_{N,2}$ follow from Stirling's formula.
\end{proof}

\subsection{Critical-edge bounds for the contour integrals}
\label{subsec:estimate-ZNsph}

We now analyze the contour integral representations obtained in Lemmas~\ref{lem:contour-integral-1} and \ref{lem:spherical-two-contour}.
Our goal is to prove an upper bound for $Z_N^\sph$ and a lower bound for $Z_{N,2}^\sph$ on a typical event for $\bW\sim\GOE(N)$.
Combined, these estimates will imply Proposition~\ref{ppn:positive-probability-lower-bound}.

We begin with the one-replica partition function $Z_N^\sph$.

\begin{ppn}\label{ppn:denominator-upper-bound}
    There exists a universal constant $C>0$ such that with probability at least $0.9$,
    \[
        Z_N^\sph
        \le
        C N^{-1/6}e^{-N/2}
        \exp\left(\fr N2G_\bLambda(\gamma)\right).
    \]
\end{ppn}

\begin{proof}
    By Lemma~\ref{lem:contour-integral-1},
    \[
        Z_N^\sph
        =
        \fr{C_N}{2\pi i}
        \int_{\gamma-i\infty}^{\gamma+i\infty}
        \exp\left\{
            \fr N2G_\bLambda(z)
        \right\}\de z.
    \]
    Define
    \[
        I_\bLambda
        :=
        \fr{1}{2\pi i}
        \int_{\gamma-i\infty}^{\gamma+i\infty}
        \exp\left\{
            \fr N2\bigl(G_\bLambda(z)-G_\bLambda(\gamma)\bigr)
        \right\}\de z.
    \]
    Then
    \[
        Z_N^\sph
        =
        C_N
        \exp\left\{
            \fr N2G_\bLambda(\gamma)
        \right\}
        I_\bLambda .
    \]
    We claim that $I_\bLambda > 0$ almost surely, and for a universal constant $C>0$, 
    \beq\label{e:denominator-ub-goal}
      \bbP(I_\bLambda \le CN^{-2/3}) \ge 0.9\,.
    \eeq
    Together with the asymptotics of $C_N$ from Lemma~\ref{lem:contour-integral-1} this implies the conclusion.

    We turn to the proof of \eqref{e:denominator-ub-goal}.
    Write $a_k=\gamma-\lambda_k>0,k=1,\dots,N$. 
    Parametrize the contour by $z=\gamma+it$.
    Since
    \[
    G_\bLambda(\gamma+it)-G_\bLambda(\gamma)=it-\fr{1}{N}\sum_{k=1}^N\log\left(1+\fr{it}{a_k}\right),
    \]
    we have
    \[
        I_\bLambda
        =
        \fr{1}{2\pi}
        \int_{-\infty}^{\infty}
        e^{iNt/2}
        \prod_{k=1}^N
        \left(
            1+\fr{it}{a_k}
        \right)^{-1/2}
        \,\de t .
    \]
    Let $\xi_1,\dots,\xi_N$ be i.i.d. standard Gaussian random variables, conditional on $\bW$, and set
    \[
        \Xi_\bW
        =
        \sum_{k=1}^N
        \fr{\xi_k^2}{a_k}.
    \]
    Let $\bE$ (as opposed to $\bbE$) denote expectation over the $\xi_k$ conditional on $\bW$.
    Since $G_\bLambda'(\gamma)=0$, we have
    \[
        \fr1N\sum_{k=1}^N \fr1{a_k}=1.
    \]
    This means that $\bE [\Xi_\bW] =N$.
    Moreover, the characteristic function of $\Xi_\bW$ (conditional on $\bW$) is
    \[
        \bE e^{it\Xi_\bW/2}
        =
        \prod_{k=1}^N
        \left(
            1-\fr{it}{a_k}
        \right)^{-1/2}.
    \]
    Therefore, by Fourier inversion and change of variables $u=-t/2$,
    \[
        I_\bLambda
        =
        2\rho_\bW(N),
    \]
    where $\rho_\bW$ denotes the density of $\Xi_\bW$.
    This shows that $I_\bLambda > 0$ almost surely.

      We now bound the density $\rho_\bW$.
      Write
    \[
        \Xi_\bW
        =
        X
        +
        Y,\quad X=\fr{\xi_1^2}{a_1}
        +
        \fr{\xi_2^2}{a_2},\quad
        Y=\sum_{k=3}^N\fr{\xi_k^2}{a_k}\,.
    \]
    Since $X,Y$ are independent and convolution with a probability density does not increase the $L^\infty$ norm,
    \[
        \|\rho_\bW\|_\infty=\|\rho_{X}\ast \rho_Y\|_\infty
        \le
        \left\|
        \rho_{X}
        \right\|_\infty .
    \]
   Note that $X={\xi_1^2}/{a_1}+{\xi_2^2}/{a_2}$ and $\xi_1,\xi_2$ are independent.
Also, the random variable $\xi_k^2/a_k$ has density
\[
    f_k(x)=\frac{\sqrt{a_k}}{\sqrt{2\pi x}}e^{-a_k x/2},
    \qquad x>0.
\]
The density of $X$ is the convolution $f_1*f_2$.
Therefore, for $x>0$,
\[
\begin{aligned}
    \rho_X(x)
    &=
    \int_0^x
    \frac{\sqrt{a_1}}{\sqrt{2\pi y}}
    e^{-a_1y/2}
    \frac{\sqrt{a_2}}{\sqrt{2\pi(x-y)}}
    e^{-a_2(x-y)/2}
    \,\de y  \\
    &=
    \frac{\sqrt{a_1a_2}}{2\pi}
    e^{-a_2x/2}
    \int_0^x
    e^{-(a_1-a_2)y/2}
    \frac{\de y}{\sqrt{y(x-y)}}  \le
    \frac{\sqrt{a_1a_2}}{2\pi}
    \int_0^x
    \frac{\de y}{\sqrt{y(x-y)}}.
\end{aligned}
\]
    Since
    \[
        \int_0^x
        \fr{\de y}{\sqrt{y(x-y)}}
        =
        \pi ,
    \]
    we obtain that
    \[
        \|\rho_\bW\|_\infty
        \le\|\rho_X\|_\infty\le
        \fr12\sqrt{a_1a_2}.
    \]
    By \cite[Proposition~3.1]{landon2022free}, there exists a universal $\tilde{C} > 0$ such that
    \[
      \bbP(a_1 \le \tilde{C} N^{-2/3})
      = \bbP(\gamma - \lambda_1 \le \tilde{C} N^{-2/3}) \ge 0.95\,.
    \]
    By \cite[Theorem~4.5.42]{anderson2010introduction}, after possibly adjusting $\tilde{C}$ we also have
    \[
      \bbP(\lambda_1 - \lambda_2 \le \tilde{C} N^{-2/3}) \ge 0.95\,.
    \]
    On the intersection of these two events, which holds with probability at least $0.9$, we have $\max(a_1,a_2) \le 2\tilde{C} N^{-2/3}$.
   Then
   \[
   I_\bLambda=2\rho_\bW(N)\le 2\|\rho_\bW\|_\infty\le 2\tilde{C} N^{-2/3}\,.
   \]
   This proves \eqref{e:denominator-ub-goal} and the result follows.
\end{proof}

We now turn to the two-replica partition function $Z_{N,2}^{\sph}$.

\begin{ppn}\label{prop:numerator-lower-bound}
    There exists a universal constant $c>0$ such that with probability at least $0.9$,
    \[
        Z_{N,2}^{\sph}
        \ge
        c N^{-1/2} e^{-N}
        \exp(N G_\bLambda(\gamma)).
    \]
\end{ppn}
The proof of Proposition~\ref{prop:numerator-lower-bound} follows the same general strategy as that of Proposition~\ref{ppn:denominator-upper-bound}.
Using the contour integral representation, we express $Z_{N,2}^{\sph}$ as an explicit exponential factor times the density of a random vector evaluated at its mean.
The main new difficulty is that we now need a lower bound on this density, which requires a more delicate argument.

    By Lemma~\ref{lem:spherical-two-contour}, we can write
    \[
        Z_{N,2}^{\sph}
        =C_{N,2}\exp\left({N}G_{\bLambda}(\gamma)\right)I_{\bLambda,2},
    \]
    where 
    \begin{align*}
        I_{\bLambda,2}=&\ \fr{1}{(2\pi i)^3}
        \int_{(\gamma,\gamma,0)+(i\bbR)^3}
        \exp\left\{
            \fr N2
            (G_{\bLambda,2}(z_1,z_2,w)-2G_{\bLambda}(\gamma))
        \right\}
        \de z_1\,\de z_2\,\de w .
    \end{align*}
    Parameterizing $(z_1,z_2,w)$ by $(\gamma+it_1,\gamma+it_2,is)$, a straightforward computation yields that (recall $a_k=\gamma-\lambda_k,1\le k\le N$)
    \[
           I_{\bLambda,2} = \fr{1}{(2\pi)^3}\int_{\bbR^3}\exp\left(\fr{iN(t_1+t_2)}{2}\right)\prod_{k=1}^N\left(1+\fr{i(t_1+t_2)}{a_k}+\fr{-t_1t_2+s^2}{a_k^2}\right)^{-1/2}\,\de t_1\,\de t_2\,\de s\,.
    \]

Let $(g_k,h_k)_{k=1}^N$ be i.i.d. standard Gaussian vectors in $\bbR^2$, independent of $\bW$.
    Similarly to above, let $\bE$ denote expectation over $(g_k,h_k)_{k=1}^N$, conditional on $\bW$.
    Define
    \[
        T_1
        :=
        \sum_{k=1}^N\fr{g_k^2}{a_k},
        \qquad
        T_2
        :=
        \sum_{k=1}^N\fr{h_k^2}{a_k},
        \qquad
    T_3
        :=
        2\sum_{k=1}^N\fr{g_k h_k}{a_k}
    \]
    and the random vector
    \[
        \Xi_{\bW,2}=
        (T_1,T_2,T_3).
    \]
    Clearly, $\bE[\Xi_{\bW,2}]=(N,N,0)$.
    Moreover, for each $1\le k\le N$,
    \[
        \bE
        \exp\left\{
            -\fr{i}{2a_k}
            \left(
                t_1g_k^2+t_2h_k^2+2sg_kh_k
            \right)
        \right\}
        =
        \det\left(
            I+
            \fr{i}{a_k}
            \begin{pmatrix}
                t_1&s\\
                s&t_2
            \end{pmatrix}
        \right)^{-1/2},
    \]
    and therefore
    \[
\bE\exp\left\{
            -\fr i2
            \bigl(
                t_1T_1+t_2T_2+sT_3
            \bigr)
        \right\}
        =
        \prod_{k=1}^N
        \left(
            1+\fr{i(t_1+t_2)}{a_k}
            +
            \fr{-t_1t_2+s^2}{a_k^2}
        \right)^{-1/2}.
    \]
    Hence
    \[
        I_{\bLambda,2}
        =
        \fr1{(2\pi)^3}
        \int_{\bbR^3}
        e^{\fr i2(t_1N+t_2N)}
        \bE
        \exp\left\{
            -\fr i2
            \bigl(
                t_1T_1+t_2T_2+sT_3
            \bigr)
        \right\}
        \de t_1\,\de t_2\,\de s .
    \]
    Let $\rho_{\bW,2}$ denote the density of the random vector $\Xi_{\bW,2}$ (conditional on $\bW$).
    By Fourier inversion and the change of variables $u_1=-t_1/2, u_2=-t_2/2, u_3=-s/2,$ we obtain
    \[
        I_{\bLambda,2}
        =
        8\,\rho_{\bW,2}(N,N,0)=8\rho_{\bW,2}(\bE\Xi_{\bW,2}).
    \]
It remains to establish a suitable lower bound of $\rho_{\bW,2}(\bE\Xi_{\bW,2})$.
The key idea is to decompose $\Xi_{\bW,2}=U+V,$ where $U$ contains the contribution of the top $K$ eigenvalues and $V$ contains the contribution of the remaining bulk eigenvalues.
Here $K$ is a sufficiently large constant to be chosen later.

The two terms play complementary roles.
We will show that $V$ is highly concentrated on the scale $N^{2/3}$, so that its contribution amounts to a small random perturbation.
On the other hand, we will prove that $U$ admits a density which is uniformly bounded below by $\Omega(N^{-2})$ on a suitable region of size $N^{2/3}$.

A subtle point is that the density of $U$ is positive only on a restricted region rather than on all of $\bbR^3$.
We therefore choose $K$ sufficiently large so that with probability at least $0.9$ over $\bW$, the effective range of fluctuations of $V$ is contained inside this region, where the $\Omega(N^{-2})$ density lower bound for $U$ is available.
Combining the concentration of $V$ with the density lower bound for $U$ then yields the desired lower bound for the density of $\Xi_{\bW,2}$.

To implement the above strategy, we introduce a good event on $\bW$.
For an integer $K>0$ and a constant $A>0$, consider the event $\cG_{K,A}$ that $\bW$ satisfies
    \[
    A^{-1}N^{-2/3}\le a_1<\cdots<a_K\le AN^{-2/3}\,,\quad \sum_{k=1}^K\fr{1}{a_k}\ge 3N^{2/3}\,,\quad
    \sum_{k=K+1}^N\fr{1}{a_k^2}\le \fr{N^{4/3}}{100}\,.
    \]
We will show that the desired density lower bound holds under $\cG_{K,A}$ for any $K,A$, and there exists $K,A$ such that $\cG_{K,A}$ happens with probability at least $0.9$, as stated in the next two lemmas.

\begin{lem}\label{lem:density-lower-bound}
    For any integer $K\ge 2$ and any $A>0$, there exists $\tilde{c}=\tilde{c}(K,A)>0$ such that for any $\bW\in \cG_{K,A}$, $$\rho_{\bW,2}(\bE\Xi_{\bW,2})\ge \tilde{c}N^{-2}. $$
\end{lem}

\begin{lem}\label{lem:typical-event}
    There exists integer $K\ge 2$ and $A>0$ such that $\bbP_{\bW\sim\GOE(N)}(\bW\in \cG_{K,A})\ge 0.9$.
\end{lem}

With the above two lemmas we may finish the proof of Proposition~\ref{prop:numerator-lower-bound}, and thus conclude Proposition~\ref{ppn:positive-probability-lower-bound} and hence Proposition~\ref{p:J-of-zero}. 

\begin{proof}[Proof of Proposition~\ref{prop:numerator-lower-bound}]
By Lemmas~\ref{lem:density-lower-bound} and \ref{lem:typical-event}, there is a universal constant $\tilde{c}>0$, and an event with probability at least $0.9$ under which $\rho_{\bW,2}(\bE\Xi_{\bW,2})\ge \tilde{c}N^{-2}$.
Combining with the asymptotics of $C_{N,2}$ from Lemma~\ref{lem:spherical-two-contour}, we see that under the same event,
\[
Z_{N,2}^\sph=C_{N,2}\exp(NG_\bLambda(\gamma))\cdot 8\rho_{\bW,2}(\bE\Xi_{\bW,2})\ge cN^{-1/2}e^{-N}\exp(NG_\bLambda(\gamma))\,,
\]
where $c=2^{7/2}\pi^{3/2}\tilde{c}>0$.
The desired result follows.
\end{proof}

We now provide the proof of Lemmas~\ref{lem:density-lower-bound} and \ref{lem:typical-event}. 

\begin{proof}[Proof of Lemma~\ref{lem:density-lower-bound}]
Fix $K,A$, and assume $\bW\in \cG_{K,A}$.
Write
\[
\Xi_{\bW,2}=U+V=:\left(\sum_{k=1}^K\fr{g_k^2}{a_k},\ \sum_{k=1}^K\fr{h_k^2}{a_k},\ \sum_{k=1}^K\fr{2g_kh_k}{a_k}\right)+\left(\sum_{k=K+1}^N\fr{g_k^2}{a_k},\ \sum_{k=K+1}^N\fr{h_k^2}{a_k},\ \sum_{k=K+1}^N\fr{2g_kh_k}{a_k}\right)\,,
\]
and denote $\rho_{U},\rho_{V}$ as the density of $U,V$.
Since $U,V$ are independent, we have $\rho_{\bW,2}=\rho_{U}\ast \rho_{V}$.

We first control $\rho_{V}$.
Write $V:=(T_1^{>K},T_2^{>K},T_3^{>K})$.
Note that under $\cG_{K,A}$,
\[
\Var(T_1^{>K})=\Var(T_2^{>K})=\sum_{k=K+1}^N\fr{2}{a_k^2}\le \fr{N^{4/3}}{50}\,,\quad \Var(T_3^{>K})=\sum_{k=K+1}^N\fr{4}{a_k^2}\le \fr{N^{4/3}}{25}\,.
\]
Applying Chebyshev's inequality, we see that 
\[
\bbP\left(\|V-\bE V\|\ge N^{2/3}\right)\le \sum_{j=1}^3\bbP\left((T_j^{>K}-\bE T_j^{>K})^2\ge \fr{N^{4/3}}{3}\right)\le \sum_{j=1}^3\fr{9\Var(T_j^{>K})}{N^{4/3}}\le \fr{18}{25}\,,
\]
and thus,
\[
\bbP\left(V\in \oB(\bE V,N^{2/3})\right)=\int_{\oB(\bE V,N^{2/3})}\rho_{V}(x)\,\de x\ge \fr{7}{25}\,,
\]
where $\oB(x_0,r)\subset \bbR^3$ is the closed ball centered at $x_0$ with radius $r$.

On the other hand, we claim that there exists $c_1=c_1(K,A)>0$, such that under $\cG_{K,A}$,
\[
\rho_{U}(x)\ge c_1N^{-2}\,,\quad\forall x\in \oB(\bE U,N^{2/3})\,.
\]
The proof proceeds via a compactness argument.
After rescaling, the law of $U$ can be viewed as a member of a compact family of distributions indexed by the finite-dimensional parameter vector $(b_1,\dots,b_K)$.
We show that each such distribution admits a density that is strictly positive near the origin.
A standard compactness argument then yields a uniform positive lower bound over the entire family.

Precisely, note that for any vector $b=(b_1,\dots,b_K)\in \bbR_+^K$, the random vector
\[
\left(\sum_{k=1}^Kb_k(g_k^2-1),\sum_{k=1}^Kb_k(h_k^2-1),\sum_{k=1}^K 2b_kg_kh_k\right)
\]
has a density $\tilde{\rho}_{b}$ that is positive in the region
\[
\Omega_{b}=\{(x,y,z)\in \bbR^3: x+m > 0, y+m > 0, \left(x+m\right)\left(y+m\right)>z^2/4\}\,,
\]
where $m=m(b)=\sum_{k=1}^K{b_k}$.
Indeed, the random vector can be written as
\[
\left(
\sum_{k=1}^K b_k g_k^2,\,
\sum_{k=1}^K b_k h_k^2,\,
\sum_{k=1}^K 2b_k g_k h_k
\right)
-(m,m,0),
\]
and the first term is precisely the collection of entries of the random positive semidefinite matrix
\[
\sum_{k=1}^K b_k
\begin{pmatrix}
g_k\\ h_k
\end{pmatrix}
\begin{pmatrix}
g_k&\!h_k
\end{pmatrix}.
\]
Every point of $\Omega_{b}$ corresponds to a positive definite matrix.
Such a matrix admits a preimage $(g_k,h_k)$ of the above map, with $(g_k),(h_k)$ linearly independent (here we use $K\ge 2$).
Moreover, one checks that the map is a submersion at the preimage.
Since $(g_k,h_k)_{k=1}^K$ has a smooth density which is strictly positive on $\bbR^{2K}$, by the coarea formula, the pushforward measure admits a density that is strictly positive on $\Omega_{b}$.

Next, we note that as long as $m(b)\ge 3$, $\Omega_{b}\supset \oB(0,1)$.
This implies that whenever $m(b)\ge 3$, we have $\tilde{\rho}_{b}(y)>0$ for any $y\in \oB(0,1)$.
Additionally, it is easy to check that the function
\[
b\mapsto \min_{y\in \oB(0,1)}\tilde{\rho}_{b}(y)
\]
is continuous with respect to $b\in \bbR_+^K$, hence its minimum over the compact set
\[
\overline{\Omega}_{K,A}=:\{b\in \bbR_+^K:b_i\in [A^{-1},A],1\le i\le K,m(b)\ge 3\}
\]
is strictly positive, which we denote as $c_1$.
Note that $c_1$ only depends on $K,A$.

Finally, for $\bW\in \cG_{K,A}$, let $b_k=(N^{2/3}a_k)^{-1}$.
Then $b=(b_1,\dots,b_K)\in \overline{\Omega}_{K,A}$.
For any $x\in \oB(\bE U,N^{2/3})$, by rescaling of $N^{-2/3}$, it is easy to see that
\[
\rho_{U}(x)=N^{-2}\tilde{\rho}_{b}(N^{-2/3}(x-\bE U))\,.
\]
Since $y:=N^{-2/3}(x-\bE U)\in \oB(0,1)$, the above is lower bounded by
\[
N^{-2}\inf_{b\in \overline{\Omega}_{K,A}}\min_{y\in \oB(0,1)}\tilde{\rho}_{b}(y)= c_1N^{-2}\,,
\]
and the claim follows.

Combining the inputs together, we obtain
\begin{align*}
    \rho_{\bW,2}(\bE\Xi_{\bW,2})
    =(\rho_{U}\ast \rho_{V})(\bE U+\bE V)
    \ge \int_{\oB(0,N^{2/3})}\rho_{U}(\bE U+x)\rho_{V}(\bE V-x)\,\de x
    \ge \fr{7c_1}{25}N^{-2}\,,
\end{align*}
and the result follows by taking $\tilde{c}=7c_1/25$. 
\end{proof}

The proof of Lemma~\ref{lem:typical-event} is essentially a combination of standard random matrix theory inputs. 

\begin{proof}[Proof of Lemma~\ref{lem:typical-event}]
We use the following standard inputs.
First, by the edge tightness of $\lambda_1$ \cite[Theorem~4.5.42]{anderson2010introduction} and \cite[Proposition~3.1]{landon2022free}, for $C$ sufficiently large,
\[
    \bbP(\cG_1\cap \cG_2)\ge 0.99,
\]
where
\[
    \cG_1
    =
    \{2-CN^{-2/3}\le \lambda_1\le 2+CN^{-2/3}\},
\]
and
\[
    \cG_2
    =
    \{C^{-1}N^{-2/3}\le \gamma-\lambda_1\le CN^{-2/3}\}.
\]

Second, we use the GOE edge counting estimate (see \cite[Proposition 6.5]{landon2022fluctuations}): there is a universal constant $C_0>0$ such that, uniformly for $s\in [C_0,N^{4/15}]$,
\[
    \left|
    \bbE \#\{k:\lambda_k\ge 2-sN^{-2/3}\}
    -
    \fr{2}{3\pi}s^{3/2}
    \right|
    \le C_0,
\]
and
\[
    \Var
    \#\{k:\lambda_k\ge 2-sN^{-2/3}\}
    \le C_0\log s .
\]

Now we fix an integer $K$ such that $$K^{2/3}\ge \max(C_0,2C), \quad K\ge 10C_0, \quad \sum_{k\ge K}\fr{\log k}{k^2}\le \fr{1}{500C_0},\quad \fr{K}{2C+10K^{2/3}}\ge 3\,,\quad \sum_{k>K}k^{-4/3}\le \fr{1}{800}\,.$$
We first control $\lambda_K$.
Define
\[
    \cG_3
    =
    \left\{
    K^{2/3}N^{-2/3}
    \le 2-\lambda_K
    \le
    10K^{2/3}N^{-2/3}
    \right\}.
\]
Indeed, since $K^{2/3}\ge C_0$, we may apply the counting estimate and obtain that
\[
    \bbE\#\{k:\lambda_k\ge 2-K^{2/3}N^{-2/3}\}
    \le
    \fr{2}{3\pi}K+C_0<\fr{1}{2}K\,,
\]
while
\[
    \bbE\#\{k:\lambda_k\ge 2-10K^{2/3}N^{-2/3}\}
    \ge
    \fr{2}{3\pi}10^{3/2}K-C_0>2K,
\]
where the strict inequalities follow from $K>10C_0$.
By Chebyshev's inequality and the variance bound,
\[
    \bbP(\cG_3^c)
    \le \bbP\big(\#\{k:\lambda_k\ge 2-K^{2/3}N^{-2/3}\}\ge K\big)+\bbP\big(\#\{k:\lambda_k\ge 2-10K^{2/3}N^{-2/3}\}\le K\}\big)\le
    \fr{5C_0\log K}{K^2}\le 0.01.
\]

Next define the rigidity event
\[
    \cG_4
    =
    \left\{
    2-\lambda_k\ge N^{-2/3}k^{2/3}
    \text{ for all } K\le k\le N^{99/100}
    \right\}.
\]
For $K\le k\le N^{0.01}$, similarly as above, the counting estimate and Chebyshev's inequality imply
\[
    \bbP\left(2-\lambda_k<N^{-2/3}k^{2/3}\right)
    \le \fr{5C_0\log k}{k^2},
\]
and hence
\[
    \sum_{K\le k\le N^{0.01}}
    \bbP\left(2-\lambda_k<N^{-2/3}k^{2/3}\right)
    \le
    \sum_{k\ge K}\fr{5C_0\log k}{k^2}\le 0.01,
\]
where the last inequality follows by our choice of $K$.
Let $\mu_k$ denote the classical location of the $k$-th eigenvalue of $\bW$, i.e. $\mu_k \in [-2,2]$ satisfies
\[
  \int_{\mu_k}^2 \fr{\sqrt{4-x^2}}{2\pi} \,\de x = \fr{k}{N}\,.
\]
For $N^{0.01}\le k\le N^{99/100}$, the usual rigidity estimate from \cite[Theorem 2.2]{ErdosYauYin2012Rigidity} implies that with probability at least $0.99$ for large $N$, 
\[
2-\lambda_k=(1+o(1))(2-\mu_k)=(1+o(1))\left(\fr{3\pi}{2}\right)^{2/3}N^{-2/3}k^{2/3}\,,\quad \text{for all }N^{0.01}\le k\le N^{99/100}\,,
\]
and thus $\bbP[\cG_4]\ge 0.98$.

Consequently, for our choice of $C$ and $K$, it holds for large $N$ that
\[
    \bbP(\cG_1\cap\cG_2\cap\cG_3\cap\cG_4)
    \ge
    1-0.01-0.01-0.02
   >0.9.
\]
We now verify the three defining properties of $\cG_{K,A}$ on this event.
This will conclude the lemma.

First, by $\cG_2$,
\[
    a_1=\gamma-\lambda_1\ge C^{-1}N^{-2/3}.
\]
Also,
\[
    a_K
    =
    \gamma-\lambda_K
    =
    (\gamma-\lambda_1)+(\lambda_1-2)+(2-\lambda_K).
\]
Using $\cG_1,\cG_2,\cG_3$, we get
\[
    a_K
    \le
    (2C+10K^{2/3})N^{-2/3}\,.
\]
Therefore, taking $A=2C+10K^{2/3}$, it holds that
\[
    A^{-1}N^{-2/3}
    \le
    a_1<\cdots<a_K
    \le
    AN^{-2/3}.
\]
Second, by monotonicity we have
\[
    \sum_{k=1}^K\fr1{a_k}
    \ge
    \fr{K}{a_K}
    \ge
    \fr{K}{(2C+10K^{2/3})N^{-2/3}}\ge 3N^{2/3},
\]
where the last inequality follows by our choice of $K$. 

Finally, for every $k\ge K$, since $\gamma>\lambda_1$,
\[
    a_k=\gamma-\lambda_k\ge \lambda_1-\lambda_k.
\]
On $\cG_1\cap\cG_4$, for every $K\le k\le N^{99/100}$,
\[
    \lambda_1-\lambda_k
    =
    (2-\lambda_k)-(2-\lambda_1)
    \ge
    N^{-2/3}k^{2/3}-CN^{-2/3}\ge \fr{1}{2}N^{-2/3}k^{2/3},
\]
where the last inequality follows as $k^{2/3}\ge K^{2/3}\ge 2C$. 
Moreover, for $k>N^{99/100}$ we have $2-\lambda_k>\fr{1}{2}N^{-2/3}N^{33/50}>40N^{-1/6}$ for large $N$.
On the event $\cG_1$, this implies
\[
	a_k \ge \lambda_1 - \lambda_k
	= (2 - \lambda_k) - (2 - \lambda_1)
	\ge 40N^{-1/6} - CN^{-2/3}
	\ge 20N^{-1/6}\,.
\]
Hence, 
\[
    \sum_{k=K+1}^N\fr1{a_k^2}
    \le
    4N^{4/3}\sum_{k=K+1}^{N^{99/100}} k^{-4/3}+N\cdot \fr{N^{1/3}}{400}
    \le
   \fr{N^{4/3}}{100}\,,
\]
where the last inequality follows from our choice of $K$.
This completes the proof.
\end{proof}

\begin{proof}[Proofs of Propositions~\ref{p:J-of-zero} and \ref{ppn:positive-probability-lower-bound}]
  Proposition~\ref{ppn:positive-probability-lower-bound} follows by considering the event in the intersection of Propositions~\ref{ppn:denominator-upper-bound}--\ref{prop:numerator-lower-bound}.
  Proposition~\ref{p:J-of-zero} then follows from \eqref{e:J-of-zero-rotational-invariance}, Proposition~\ref{ppn:positive-probability-lower-bound}, and positivity of $Z^\sph_N$, $Z^\sph_{N,2}$.
\end{proof}

\section{Proof of main variance estimate}
\label{s:variance}

In this section, we complete the proof of Theorem~\ref{t:main}.
This section is structured as follows.
\begin{itemize}
  \item In \S\ref{ss:overlap-ub-24mt} we prove Proposition~\ref{p:overlap-ub-24mt}, which provides asymptotically sharp upper bounds on $\bbE \la R_{1,2}^2 \ra$ and $\bbE \la R_{1,2}^4 \ra$.
  This is the part of the proof of Theorem~\ref{t:main} that is new to this paper; Theorem~\ref{t:main} follows by combining it with known results from \cite{dey2026fluctuations} and \cite{chatterjee2009disorder}.
  \item In \S\ref{ss:cavity} we describe the input from \cite{dey2026fluctuations}.
  The main result of this subsection is Proposition~\ref{p:cavity}, which gives a two-sided bound on a correlated overlap moment $\bbE \la R_{1,2}^2 \ra_t$, where the replicas $\bx^1,\bx^2$ are Gibbs samples of two $t$-correlated SK Hamiltonians, in terms of $\bbE \la R_{1,2}^2 \ra$ and $\bbE \la R_{1,2}^4 \ra$.

  \item In \S\ref{ss:variance} we complete the proof of Theorem~\ref{t:main} using an integral formula for $\Var(F_N)$ in terms of the $\bbE \la R_{1,2}^2 \ra_t$ due to \cite{chatterjee2009disorder}.
  \item In \S\ref{ss:variance-consequences} we obtain Theorem~\ref{t:overlap}\ref{i:overlap-lb} and Corollary~\ref{c:sphere}\ref{i:sphere-overlap-lb} as simple consequences of this proof.
\end{itemize}

\subsection{Annealed overlap upper bound}
\label{ss:overlap-ub-24mt}

The main result of this subsection is the following proposition, which controls the second and fourth annealed overlap moments.
Recall $\eps = 0.01$. 
\begin{ppn}\label{p:overlap-ub-24mt}
  We have
  \balnn
    \bbE \la R_{1,2}^2 \ra &\lesssim \max\lt((NJ(0)^2)^{-1}, N^{-2/3-\eps}\rt)\,, & 
    \bbE \la R_{1,2}^4 \ra &\lesssim N^{-4/3}\,.
  \ealnn
  Together with the bound $J(0) \gtrsim N^{-1/6}$ from Proposition~\ref{p:J-of-zero}, the first estimate implies $\bbE \la R_{1,2}^2 \ra \lesssim N^{-2/3}$.
\end{ppn}
\begin{rmk}\label{r:J-of-zero-explicit-dependence}
  The estimate on $\bbE \la R_{1,2}^2 \ra$ above keeps the dependence on $J(0)$ explicit so that, after we prove $\bbE \la R_{1,2}^2 \ra\gtrsim N^{-2/3}$ (Theorem~\ref{t:overlap}\ref{i:overlap-lb}), we can infer the matching upper bound $J(0) \lesssim N^{-1/6}$.
  This is useful for the proof of Corollary~\ref{c:sphere}\ref{i:sphere-overlap-lb} in \S\ref{ss:variance-consequences}.
\end{rmk}

\begin{cor}\label{c:reweight-annealed-overlap-ub}
  Recall $X_N = Z_N / Z^\sph_N$.
  For $k=1,2$ we have
  \[
    \bbE [X_N^2 \la R_{1,2}^{2k} \ra] \lesssim (N^{1/2} J(0))^{-2k}\,.
  \]
  Furthermore, there exists a universal constant $c>0$ such that for any $t\ge 0$,
  \[
    \bbE[X_N^2 \la \bone\{N^{1/3} |R_{1,2}| > t\} \ra] \le 2\exp(-ct)\,.
  \]
\begin{proof}
  Immediate from Propositions~\ref{p:J-of-zero} and \ref{p:reweight-annealed-overlap-ub}.
\end{proof}
\end{cor}
\begin{lem}[{\cite[Theorem 3]{chen2023gaussian}}]\label{l:convex-gaussian-lower-tail}
  Suppose $\bg \sim \cN(0,\bI_n)$ and $F:\bbR^n\rightarrow\bbR$ is convex with $\bbE[F(\bg)^2] < \infty$.
  Then, for $M \in \bbR$ the median of $F(\bg)$ and any $t\ge 0$,
  \[
    \bbP(F(\bg) \le M - t) \le \exp\lt(-\fr{t^2}{2\Var(F(\bg))}\rt)\,.
  \]
\end{lem}
\begin{lem}\label{l:fe-lower-tail-crude}
  There exists a universal constant $C>0$ such that for all $t\ge 0$,
  \[
    \bbP\lt(F_N \le \fr{N}{4} - \fr{\log N}{12} - t\rt) \le 2\exp\lt(-\fr{t^2}{C\log^2 N}\rt)\,.
  \]
\begin{proof}
  We can view $F_N = F_N(\bW)$ as a convex function of the i.i.d. standard Gaussians $(g_{i,j})_{1\le i\le j\le N}$
  \balnn
    g_{i,i} &= \sqrt{N/2} \cdot W_{i,i}\,, & 
    g_{i,j} &= \sqrt{N} \cdot W_{i,j}\,.
  \ealnn
  Let $M$ be the median of $F_N$, $m = \fr{N}{4} - \fr{\log N}{12}$, and $\Delta = |M-m|$.
  By Theorem~\ref{t:compare-cube-sphere} and Chebyshev's inequality, $\bbP(|F_N - F^\sph_N| \ge 1) \lesssim N^{-1/3}$.
  Therefore, there exists a universal constant $C$ such that for all $t\in \bbR$,
  \[
  	\bbP(F^\sph_N \ge t + 1) - CN^{-1/3} \le \bbP(F_N \ge t) \le \bbP(F^\sph_N \ge t - 1) + CN^{-1/3}\,.
  \]
  Together with the spherical model's CLT \eqref{e:spherical-critical-clt} this implies $\Delta = o(\sqrt{\log N})$. 
  Meanwhile, by Theorem~\ref{t:CL19}, there exists a universal constant $C'$ such that $\Var(F_N) \le C' \log^2 N$.
  By Lemma~\ref{l:convex-gaussian-lower-tail},
  \[
    \bbP(F_N \le m - t)
    \le \bbP(F_N \le M - (t - \Delta))
    \le \exp\lt(-\fr{(t - \Delta)_+^2}{2C' \log^2 N}\rt)\,.
  \]
  We will show the result holds with $C = 8C'$.
  If $t \ge \sqrt{\log N}$, then $(t - \Delta)_+ \ge t/2$, so
  \[
    \exp\lt(-\fr{(t - \Delta)_+^2}{2C' \log^2 N}\rt)
    \le \exp\lt(-\fr{t^2}{8C' \log^2 N}\rt)\,,
  \]
  as desired.
  If $t \le \sqrt{\log N}$, then the result is trivial because
  \[
    2\exp\lt(-\fr{t^2}{C\log^2 N}\rt)
    \ge 2 \exp\lt(-\fr{1}{C\log N}\rt) \ge 1\,. \qedhere
  \]
\end{proof}
\end{lem}

\begin{proof}[Proof of Proposition~\ref{p:overlap-ub-24mt}]
  We will prove the stronger statement that for $k=1,2$,
  \[
    \bbE \la R_{1,2}^{2k} \ra \lesssim \min\lt(N^{1/2} J(0), N^{1/3+\eps} \rt)^{-{2k}}\,.
  \]
  Let $c$ be given by Corollary~\ref{c:reweight-annealed-overlap-ub}.
  Define $\tau = N^{-(2k-1)/(6k) - \eps}$, $\eta = \exp(-\fr{c}{3} N^{1/(6k) - \eps})$, and
  \balnn
    Y &= \la R_{1,2}^{2k} \ra\,, &
    Y_{\le} &= \la R_{1,2}^{2k} \bone\{|R_{1,2}| \le \tau\} \ra\,, & 
    Y_{>} &= \la R_{1,2}^{2k} \bone\{|R_{1,2}| > \tau\} \ra\,, & 
    Y_{+} &= \la \bone\{|R_{1,2}| > \tau\} \ra\,.
  \ealnn
  We write
  \[
    \bbE Y
    = \bbE[\bone\{X_N \ge 1/2\} Y]
    + \bbE[\bone\{X_N < 1/2\} Y_{\le}]
    + \bbE[\bone\{\eta \le X_N < 1/2\} Y_{>}]
    + \bbE[\bone\{X_N < \eta\} Y_{>}]\,.
  \]
  We estimate these terms individually.
  The first assertion of Corollary~\ref{c:reweight-annealed-overlap-ub} gives
  \[
    \bbE[\bone\{X_N \ge 1/2\} Y] \le 4\bbE[X_N^2 Y] \lesssim (N^{1/2} J(0))^{-2k}\,.
  \]
  By Theorem~\ref{t:compare-cube-sphere} and Chebyshev's inequality, $\bbP(X_N < 1/2) \lesssim N^{-1/3}$.
  Since $Y_{\le} \le \tau^{2k}$ deterministically,
  \[
    \bbE[\bone\{X_N < 1/2\} Y_{\le}]
    \lesssim N^{-1/3} \tau^{2k}
    = N^{-2k(1/3 + \eps)}\,.
  \]
  Next, note that
  \[
    \bbE[\bone\{\eta \le X_N < 1/2\} Y_{>}]
    \le \eta^{-2} \bbE[X_N^2 Y_{>}]
    \le \eta^{-2} \bbE[X_N^2 Y_{+}]
    \stackrel{(*)}{\le} 2\eta^{-2} \exp(-cN^{1/3} \tau)
    = 2\eta
    \lesssim N^{-2k(1/3 + \eps)}\,,
  \]
  where the step marked $(*)$ follows from the second assertion of Corollary~\ref{c:reweight-annealed-overlap-ub}.
  Finally, let $m = \fr{N}{4} - \fr{\log N}{12}$ and $C$ be given by Lemma~\ref{l:fe-lower-tail-crude}.
  By Lemma~\ref{l:fe-lower-tail-crude} and Markov's inequality,
  \balnn
    \bbE[\bone\{X_N < \eta\} Y_{>}]
    &\le \bbP(X_N < \eta)
    \le \bbP(Z_N < e^m \eta^{1/2}) + \bbP(Z^\sph_N > e^m \eta^{-1/2}) \\
    &\le 2\exp\lt(-\fr{(\fr12 \log \fr{1}{\eta})^2}{C\log^2 N}\rt) + \fr{\bbE[Z^\sph_N]}{e^m \eta^{-1/2}} \\
    &= 2\exp\lt(-\fr{c^2 N^{1/(3k) - 2\eps}}{36C\log^2 N}\rt) + N^{1/12} \exp\lt(-\fr{cN^{1/(6k) - \eps}}{6}\rt)
    \lesssim N^{-2k(1/3 + \eps)}\,. \qedhere
  \ealnn
\end{proof}

\subsection{Cavity estimate of correlated overlap moment}
\label{ss:cavity}

The contents of this subsection are adapted essentially verbatim from \cite{dey2026fluctuations}, specialized to $\beta = 1$.
Note that while the main result of \cite{dey2026fluctuations} assumes $\beta = 1 - cN^{-1/3}$ for fixed $c>0$, the parts we cite hold for general $\beta$.
We first introduce a correlated Gibbs average $\la \cdot \ra_t$.
Since the influential work of Chatterjee \cite{chatterjee2009disorder}, controlling the free energy variance through correlated overlap moments $\bbE \la R_{1,2}^2 \ra_t$ has become a widely used strategy in the free energy fluctuations literature \cite{chen2019order,dey2026fluctuations}.
The main result of this section is Proposition~\ref{p:cavity} below, which provides a two-sided bound for $\bbE \la R_{1,2}^2 \ra_t$.
\begin{dfn}\label{d:correlated-gibbs-average}
  Let $\bW^0,\bW^1,\bW^2$ be independent copies of $\bW$, and for $t\in[0,1]$ define
  \balnn
    \bW_t^1 &= \sqrt{t}\bW^0 + \sqrt{1-t}\bW^1\,, &
    \bW_t^2 &= \sqrt{t}\bW^0 + \sqrt{1-t}\bW^2\,.
  \ealnn
  We write $\la \cdot \ra_t$ for the average over $\bx^1,\bx^2\in \Sigma_N$ sampled independently from the Gibbs measures \eqref{e:gibbs-msr} (with $\beta=1$) with disorder matrices $\bW_t^1$ and $\bW_t^2$, respectively.
  Since we fix $\beta=1$, there will be no confusion with the notation $\la \cdot \ra_\beta$ in \S\ref{s:intro}.
  We continue to use the unsubscripted notation $\la \cdot \ra$ for average over the original Gibbs measure at $t=1$.
\end{dfn}
\begin{ppn}\label{p:cavity}
  There exists a universal constant $C>0$ such that for all $t\in [0,1]$,
  \[
    \fr1N 
    - \fr2N \bbE \la R_{1,2}^2\ra 
    - C \sqrt{\bbE \la R_{1,2}^2 \ra_t \cdot \bbE \la R_{1,2}^4\ra}
    - C \bbE \la R_{1,2}^4\ra
    \le (1-t) \bbE \la R_{1,2}^2 \ra_t 
    \le \fr1N + C \bbE \la R_{1,2}^4\ra\,.
  \]
\end{ppn}
Proposition~\ref{p:cavity} is a direct consequence of several lemmas from \cite{dey2026fluctuations}, which are proved therein using Talagrand's cavity method. 
We first define several objects arising in this proof.
For $\bx^1,\bx^2 \in \Sigma_N$, define
\[
  R^-(\bx^1,\bx^2) = \fr1N \sum_{i=1}^{N-1} x^1_i x^2_i\,.
\]
\begin{dfn}\label{d:cavity-gibbs-average}
  Consider $t,s\in [0,1]$. 
  Let $\bW_t^1$, $\bW_t^2$ be as in Definition~\ref{d:correlated-gibbs-average}.
  Let $\bD_s$ be the diagonal matrix with $(\bD_s)_{N,N} = \sqrt{s}$ and all other diagonal entries equal to $1$.
  Then define
  \balnn
    \bW_{t,s}^1 &= \bD_s\bW_t^1\bD_s\,, & 
    \bW_{t,s}^2 &= \bD_s\bW_t^2\bD_s\,.
  \ealnn
  We write $\la \cdot \ra_{t,s}$ for the average over $\bx^1,\bx^2\in \Sigma_N$ sampled independently from the Gibbs measures \eqref{e:gibbs-msr} (with $\beta=1$) with disorder matrices $\bW_{t,s}^1$ and $\bW_{t,s}^2$, respectively.
  Note that $\la \cdot \ra_{t,1} = \la \cdot \ra_t$.

  We will write $\bx^{1,1},\bx^{1,2},\ldots$ and $\bx^{2,1},\bx^{2,2},\ldots$ for independent replicas of $\bx^1$ and $\bx^2$, respectively.
  In the below lemmas we abbreviate $R_{1,2} = R(\bx^1,\bx^2)$ and $R^-_{1,2} = R^-(\bx^1,\bx^2)$, but will make the replicas explicit in any Gibbs average involving more than two replicas.
\end{dfn}
\begin{lem}[{\cite[Equation 7]{dey2026fluctuations}}]\label{l:cavity-expansion}
  Let $f_1(\bx^1,\bx^2) = x^1_N x^2_N R^-(\bx^1,\bx^2)$.
  For all $t\in [0,1]$, there exists $s^* = s^*(t) \in [0,1]$ such that
  \balnn
    \bbE \la R_{1,2}^2 \ra_t
    &= \fr1N + t \bbE \la (R^-_{1,2})^2 \ra_{t,0}
    - 2t \bbE \la R^-(\bx^{1,1},\bx^2) R^-(\bx^{1,2},\bx^2) R^-(\bx^{1,1},\bx^{1,2})\ra_{t,0} \\
    &\qquad + \fr16 \cdot \fr{\partial^3}{\partial s^3} \bbE \la f_1(\bx^1,\bx^2) \ra_{t,s} \Big|_{s=s^*}\,.
  \ealnn
\end{lem}

\begin{lem}[{\cite[Lemmas 2.3--2.6]{dey2026fluctuations}}]\label{l:cavity-bounds}
  There is a universal constant $C>0$ such that for all $t\in [0,1]$,
  \balnn
    \lt|\bbE \la (R^-_{1,2})^2 \ra_t - \bbE \la (R^-_{1,2})^2 \ra_{t,0} \rt| &\le C \bbE \la R_{1,2}^4 \ra\,, \\
    0 \le \bbE \la R_{1,2}^2 \ra_t - \bbE \la (R^-_{1,2})^2 \ra_t &\le \fr{2}{N} \bbE \la R_{1,2}^2 \ra\,, \\
    0 \le \bbE \la R^-(\bx^{1,1},\bx^2) R^-(\bx^{1,2},\bx^2) R^-(\bx^{1,1},\bx^{1,2})\ra_{t,0}
    &\le C \sqrt{\bbE \la R_{1,2}^2\ra_t \cdot \bbE \la R_{1,2}^4 \ra}\,, \\
    \sup_{s\in [0,1]} \lt|\fr{\partial^3}{\partial s^3} \bbE \la f_1(\bx^1,\bx^2) \ra_{t,s}\rt| 
    &\le C \bbE \la R_{1,2}^4 \ra\,.
  \ealnn
\end{lem}

\begin{proof}[Proof of Proposition~\ref{p:cavity}]
  The conclusion of Lemma~\ref{l:cavity-expansion} rearranges to
  \balnn
    (1-t) \bbE \la R_{1,2}^2 \ra_t
    &=
    \fr1N 
    - t \lt[
      \bbE \la R_{1,2}^2 \ra_t 
      - \bbE \la (R^-_{1,2})^2 \ra_t
    \rt]
    - t \lt[
      \bbE \la (R^-_{1,2})^2 \ra_t
      - \bbE \la (R^-_{1,2})^2 \ra_{t,0}
    \rt] \\
    &\qquad- 2t \bbE \la R^-(\bx^{1,1},\bx^2) R^-(\bx^{1,2},\bx^2) R^-(\bx^{1,1},\bx^{1,2})\ra_{t,0}
    + \fr16 \cdot \fr{\partial^3}{\partial s^3} \bbE \la f_1(\bx^1,\bx^2) \ra_{t,s} \Big|_{s=s^*}\,.
  \ealnn
  These terms are bounded by Lemma~\ref{l:cavity-bounds}.
\end{proof}

\subsection{Integrating correlated overlaps}
\label{ss:variance}

We are now ready to complete the proof of  Theorem~\ref{t:main}.
We first record two useful properties of the correlated overlap moments $\bbE\la R_{1,2}^{2k} \ra_t$, due to Chatterjee.
\begin{fac}[{\cite[Theorem~3.1]{chatterjee2009disorder}}]\label{f:correlated-overlap-incr}
  For any integer $k\ge 0$, $[0,1] \ni t \mapsto \bbE\la R_{1,2}^{2k} \ra_t$ is nondecreasing.
\end{fac}
\begin{fac}[{\cite[Theorem~3.8]{chatterjee2009disorder}}]\label{f:chatterjee-var-id}
  We have $\Var(F_N) = \fr{N}{2} \int_0^1 \bbE \la R_{1,2}^2 \ra_t \,\de t$.
\end{fac}

\begin{proof}[Proof of Theorem~\ref{t:main}]
  Part \ref{i:main-clt} is a direct consequence of Theorem~\ref{t:compare-cube-sphere} and the spherical model's CLT \eqref{e:spherical-critical-clt} (with $b=0$).
  It remains to prove part \ref{i:main-var}.

  By the upper bound in Proposition~\ref{p:cavity} and the estimate $\bbE \la R_{1,2}^4 \ra \lesssim N^{-4/3}$ from Proposition~\ref{p:overlap-ub-24mt}, there exists a universal constant $C$ such that for all $t\in [0,1]$,
  \[
    \bbE \la R_{1,2}^2 \ra_t \le \fr{1 + CN^{-1/3}}{N(1-t)}\,.
  \]
  Fact~\ref{f:correlated-overlap-incr} and the estimate $\bbE \la R_{1,2}^2 \ra \lesssim N^{-2/3}$ from Proposition~\ref{p:overlap-ub-24mt} imply that for all $t\in [0,1]$,
  \[
    \bbE \la R_{1,2}^2 \ra_t \le CN^{-2/3}\,,
  \]
  after possibly adjusting the universal constant $C$.
  Let $\delta = N^{-1/3}$.
  By Fact~\ref{f:chatterjee-var-id} and the last two displays,
  \balnn
    \Var(F_N)
    &\le \fr{N}{2} \int_0^{1-\delta} \fr{1 + CN^{-1/3}}{N(1-t)} \,\de t
    + \fr{N}{2} \int_{1-\delta}^1 CN^{-2/3} \,\de t \\
    &= \fr{1}{2} (1 + CN^{-1/3}) \log \fr{1}{\delta} + \fr{C}{2}
    = \fr{1}{6} \log N + O(1)\,.
  \ealnn
  This proves the desired upper bound on $\Var(F_N)$; we turn to the matching lower bound.
  By the lower bound in Proposition~\ref{p:cavity} and the estimates $\bbE \la R_{1,2}^2 \ra \lesssim N^{-2/3}$, $\bbE \la R_{1,2}^4 \ra \lesssim N^{-4/3}$ from Proposition~\ref{p:overlap-ub-24mt}, there exists a universal constant $C$ such that for all $t\in [0,1)$,
  \beq\label{e:correlated-overlap-quadratic-inequality}
    (1-t) \bbE \la R_{1,2}^2 \ra_t + CN^{-2/3} \sqrt{\bbE \la R_{1,2}^2 \ra_t} \ge \fr{1}{N} (1 - CN^{-1/3})\,.
  \eeq
  Write $a_2 = 1-t$, $a_1 = CN^{-2/3}$, $a_0 = \fr{1}{N} (1-CN^{-1/3})$.
  Since $a_2,a_1,a_0 > 0$, the quadratic equation
  \[
    a_2 x^2 + a_1 x = a_0
  \]
  has exactly one positive solution; denote this solution $\theta^*_t$.
  Then $a_2 (\theta^*_t)^2 \le a_0$, so
  \[
    \theta^*_t
    \le (a_0 / a_2)^{1/2}
    \le \lt(N(1-t)\rt)^{-1/2}\,.
  \]
  The quadratic inequality \eqref{e:correlated-overlap-quadratic-inequality} then implies
  \baln\label{e:theta-t-lb}
    (1-t) \bbE \la R_{1,2}^2 \ra_t
    \ge (1-t) (\theta^*_t)^2
    &= \fr{1}{N} (1 - CN^{-1/3}) - CN^{-2/3} \theta^*_t \\
    \nonumber
    &\ge \fr{1}{N} \lt(1 - CN^{-1/3} - \fr{C}{N^{1/6} (1-t)^{1/2}}\rt)\,.
  \ealn
  Recall $\delta = N^{-1/3}$.
  By Fact~\ref{f:chatterjee-var-id},
  \balnn
    \Var(F_N)
    &\ge \fr{N}{2} \int_0^{1-\delta} \bbE \la R_{1,2}^2 \ra_t \,\de t
    \ge \fr12 \int_0^{1-\delta} \lt(\fr{1 - CN^{-1/3}}{1-t} - \fr{C}{N^{1/6} (1-t)^{3/2}}\rt) \,\de t \\
    &= \fr12 (1 - CN^{-1/3}) \log \fr{1}{\delta}
    - \fr{C}{N^{1/6}} (\delta^{-1/2} - 1)
    = \fr16 \log N - O(1)\,. \qedhere
  \ealnn
\end{proof}

\subsection{Additional consequence: overlap lower bounds}
\label{ss:variance-consequences}

We finally prove Theorem~\ref{t:overlap}\ref{i:overlap-lb} and Corollary~\ref{c:sphere}\ref{i:sphere-overlap-lb}.

\begin{proof}[Proof of Theorem~\ref{t:overlap}\ref{i:overlap-lb}]
  Setting $t = 1 - 4C^2N^{-1/3}$ in \eqref{e:theta-t-lb} implies
  \[
    4C^2N^{-1/3} \bbE \la R_{1,2}^2 \ra_t \ge \fr{1}{2N} (1 - 2CN^{-1/3})\,.
  \]
  Together with Fact~\ref{f:correlated-overlap-incr} this yields
  \[
    \bbE \la R_{1,2}^2 \ra
    \ge \bbE \la R_{1,2}^2 \ra_t 
    \ge \fr{N^{-2/3}}{8C^2} (1 - 2CN^{-1/3})\,. \qedhere
  \]
\end{proof}

\begin{cor}\label{c:J-of-zero-ub}
  We have $J(0) \lesssim N^{-1/6}$.
\begin{proof}
  Theorem~\ref{t:overlap}\ref{i:overlap-lb} and the estimate on $\bbE \la R_{1,2}^2 \ra$ from Proposition~\ref{p:overlap-ub-24mt} imply
  \[
    N^{-2/3} 
    \lesssim \bbE \la R_{1,2}^2 \ra 
    \lesssim \max\lt((NJ(0)^2)^{-1}, N^{-2/3-\eps}\rt)\,.
  \]
  Since $N^{-2/3-\eps} \ll N^{-2/3}$, this implies $N^{-2/3} \lesssim (NJ(0)^2)^{-1}$, which rearranges to the result.
\end{proof}
\end{cor}

\begin{proof}[Proof of Corollary~\ref{c:sphere}\ref{i:sphere-overlap-lb}]
  We will argue that for a small universal constant $c > 0$,
  \beq\label{e:sphere-overlap-lb-goal}
    \bbE \la \bone\{|R_{1,2}| \le cN^{-1/3}\} \ra^\sph \le \fr12\,.
  \eeq
  This implies the conclusion, after adjusting $c$.
  Recall the density $\rho$ of $\bbE^\sph_q$ defined in \eqref{e:sphere-density}, and the function $\psi(q) = \rho(q) J(q)$ defined in \eqref{e:psi}.
  Abbreviate $\phi(q) = \log \psi(q)$ and $\delta = cN^{-1/3}$.
  By \eqref{e:reweighted-sph-2mt},
  \[
    \bbE \la \bone\{|R_{1,2}| \le cN^{-1/3}\} \ra^\sph
    = \bbE^\sph_q [\bone\{|q| \le \delta\} J(q)]
    = \int_{-\delta}^{\delta} \rho(q) J(q) \,\de q
    = \int_{-\delta}^{\delta} \psi(q) \,\de q\,.
  \]
  Recall $\rho(0) \asymp N^{1/2}$ by Stirling's formula.
  Corollary~\ref{c:J-of-zero-ub} implies that $\psi(0) \le CN^{1/3}$ for a universal constant $C$.
  From \eqref{e:phi}, for any $q\ge 0$
  \[
    \phi'(q) = - \fr{(N-3) q}{1-q^2} + Nq + f'(q)
    \le 3q\,.
  \]
  The last inequality uses that $f'(q) \le 0$ by Lemma~\ref{l:K-log-concave}.
  Therefore for any $q\in [0,\delta]$,
  \[
    \phi(q) - \phi(0) \le \int_0^q 3t\,\de t \le \fr32 \delta^2 \le \log 2\,,
  \]
  which implies
  \[
    \psi(q) \le 2\psi(0) \le 2CN^{1/3}\,.
  \]
  By evenness of $\psi$, the same estimate holds for all $|q| \le \delta$.
  Thus,
  \[
    \bbE \la \bone\{|R_{1,2}| \le cN^{-1/3}\} \ra^\sph
    \le \int_{-\delta}^{\delta} 2CN^{1/3} \,\de q
    = 4cC\,.
  \]
  Taking $c = \fr{1}{8C}$ proves \eqref{e:sphere-overlap-lb-goal}.
\end{proof}

\section{Exponential overlap moment and spherical free energy variance}
\label{s:conclusion}

In this section we prove Theorem~\ref{t:overlap}\ref{i:overlap-ub} and Corollary~\ref{c:sphere}\ref{i:sphere-var}.
These follow from results proved earlier in the paper and tail bounds on $F_N$ and $F^\sph_N$ proved in \S\ref{ss:tail-bounds}.

\subsection{Tail bounds for free energies}
\label{ss:tail-bounds}

The following tail bound is proved by essentially the same argument as Lemma~\ref{l:fe-lower-tail-crude}, but using the now-proved sharper variance estimate from Theorem~\ref{t:main} in place of Theorem~\ref{t:CL19}.
Throughout this section we set 
\[
  m = \fr{N}{4} - \fr{\log N}{12}\,.
\]
\begin{lem}\label{l:fe-lower-tail}
  For all $t\ge 0$,
  \[
    \bbP(F_N \le m - t) \le 2\exp\lt(-\fr{t^2}{\log N}\rt)\,.
  \]
\begin{proof}
  As observed in the proof of Lemma~\ref{l:fe-lower-tail-crude}, $F_N$ is a convex function of i.i.d. standard Gaussians, so Lemma~\ref{l:convex-gaussian-lower-tail} applies.
  Let $M$ be the median of $F_N$ and $\Delta = |M-m|$.
  As argued in the proof of Lemma~\ref{l:fe-lower-tail-crude}, $\Delta = o(\sqrt{\log N})$.
  Lemma~\ref{l:convex-gaussian-lower-tail} and Theorem~\ref{t:main} imply
  \[
    \bbP(F_N \le m - t)
    \le \bbP(F_N \le M - (t - \Delta))
    \le \exp\lt(-\fr{(t - \Delta)_+^2}{2\Var(F_N)}\rt)
    \le \exp\lt(-\fr{(t - \Delta)_+^2}{\fr12 \log N}\rt)\,.
  \]
  If $t > \sqrt{\log 2 \cdot \log N}$, then $\Delta = o(\sqrt{\log N})$ implies $(t - \Delta)_+ \ge 2^{-1/2} t$, so
  \[
    \exp\lt(-\fr{(t - \Delta)_+^2}{\fr12 \log N}\rt)
    \le \exp\lt(-\fr{t^2}{\log N}\rt)\,.
  \]
  If $t \le \sqrt{\log 2 \cdot \log N}$, the result is trivial because
  \[
    2\exp\lt(-\fr{t^2}{\log N}\rt)
    \ge 1\,. \qedhere
  \]
\end{proof}
\end{lem}
\begin{lem}\label{l:fe-sph-lower-tail}
  For all $t\ge 0$,
  \[
    \bbP(F^\sph_N \le m - t) \le 4\exp\lt(-\min\lt(
      t, \fr{t^2}{4\log N}
    \rt)\rt)\,.
  \]
\begin{proof}
  Theorem~\ref{t:compare-cube-sphere} and \eqref{e:overview-goal-compare-step1} imply $\bbE[X_N^2] = 1 + O(N^{-1/3})$.
  By Markov's inequality,
  \[
    \bbP(F^\sph_N \le F_N - t)
    = \bbP(X_N^2 \ge e^{2t}) \le 2\exp(-2t)\,.
  \]
  Thus,
  \[
    \bbP(F^\sph_N \le m - t)
    \le \bbP(F^\sph_N \le F_N - t/2)
    + \bbP(F_N \le m - t/2)
    \le 2 \exp(-t) + 2\exp\lt(-\fr{t^2}{4\log N}\rt)\,. \qedhere
  \]
\end{proof}
\end{lem}
\begin{lem}\label{l:fe-upper-tail}
  For all $t\ge 0$,
  \[
    \max\Big(\bbP(F_N \ge m + t), \bbP(F^\sph_N \ge m + t)\Big) \le \exp\lt(-\lt(t - \fr{\log N}{12}\rt)_+\rt)\,.
  \]
\begin{proof}
  Immediate from Markov's inequality, as $\bbE[Z_N] = \bbE[Z^\sph_N] = e^{N/4} = e^{m + (\log N)/12}$.
\end{proof}
\end{lem}

\subsection{Exponential moment of overlap}

In this subsection we prove Theorem~\ref{t:overlap}\ref{i:overlap-ub}.
Propositions~\ref{p:J-of-zero} and \ref{p:reweight-annealed-overlap-ub} imply the existence of an absolute $c_0 > 0$ such that
\beq\label{e:reweighted-exp-moment}
  \bbE\lt[
    X_N^2 \la \exp(c_0 N^{1/3} |R_{1,2}|) \ra
  \rt] \le 2\,.
\eeq
We let $c>0$ be a small universal constant we will set later, and $\theta = c / c_0 \in (0,1)$.
\begin{dfn}\label{d:a-b-recursion}
  Let $b_0 = 0$, $a_0 = - \fr{\log 0.9}{\log N}$, and $b_1$ solve
  \[
    b_1 - \fr13 = -c\,.
  \]
  For $k\ge 1$, define $a_k$ and $b_{k+1}$ recursively by
  \balnn
    2a_k - (\theta^{-1} - 1)b_k &= -ck\,, &
    b_{k+1} - \fr{a_k}{2} + \fr{1}{12} = -c(k+1)\,.
  \ealnn
\end{dfn}
We defer the proof of the following lemma to the end of the subsection.
\begin{lem}\label{l:sequences}
  For sufficiently small $c>0$, the sequences $(a_k)_{k\ge 0}$ and $(b_k)_{k\ge 0}$ are increasing, with $a_1 \ge 2$ and $\lim_{k\rightarrow\infty} a_k = \lim_{k\rightarrow\infty} b_k = \infty$.
\end{lem}
We assume without further comment that $c$ is small enough that the conclusion of Lemma~\ref{l:sequences} holds.
Let $n$ be the smallest number such that $b_n \log N > cN^{1/3}$; this exists by Lemma~\ref{l:sequences}.
For $a,b \ge 0$ define
\balnn
  A(a,b) &= \bbE \lt[
    \bone\{X_N > N^{-a}\} \la \exp(cN^{1/3}|R_{1,2}|) \bone\{cN^{1/3}|R_{1,2}| \ge b \log N\} \ra
  \rt]\,, \\
  B(a,b) &= \bbE \lt[
    \bone\{X_N \le N^{-a}\} \la \exp(cN^{1/3}|R_{1,2}|) \bone\{cN^{1/3}|R_{1,2}| < b \log N\} \ra
  \rt]\,.
\ealnn
We emphasize that all these definitions depend implicitly on $c$, which we have not yet set, and on $N$.
The proof of Theorem~\ref{t:overlap}\ref{i:overlap-ub} is based on the following decomposition.
\begin{lem}\label{l:sequences-decomp}
  We have
  \[
    \bbE \la \exp(cN^{1/3} |R_{1,2}|) \ra \le \sum_{k=0}^{n-1} \lt(A(a_k,b_k) + B(a_k,b_{k+1}) \rt)\,.
  \]
\begin{proof}
  Since $(b_k)_{0\le k\le n}$ is increasing and $b_n \log N > cN^{1/3}$, for each realization of $(X_N, R_{1,2})$ there exists a unique $0\le k\le n-1$ such that $cN^{1/3} |R_{1,2}| \in [b_k \log N, b_{k+1}\log N)$.
  If $X_N \ge N^{-a_k}$, this point in the sample space is counted by $A(a_k,b_k)$, and if $X_N < N^{-a_k}$ it is counted by $B(a_k,b_{k+1})$.
\end{proof}
\end{lem}
The next several lemmas bound the right-hand side of Lemma~\ref{l:sequences-decomp}.
\begin{lem}\label{l:A0-term}
  For sufficiently small (but still universal) $c>0$, we have $A(a_0,b_0) \le 1.9$.
\begin{proof}
  Recall that $N^{-a_0} = 0.9$ and $b_0 = 0$.
  Thus,
  \beq\label{e:A-a0-b0}
    A(a_0,b_0) = \bbE \lt[
      \bone\{X_N > 0.9\} \la \exp(cN^{1/3}|R_{1,2}|) \ra
    \rt]
    \le (0.9)^{-2} \bbE \lt[
      X_N^2 \la \exp(cN^{1/3}|R_{1,2}|) \ra
    \rt]\,.
  \eeq
  Recall that Theorem~\ref{t:compare-cube-sphere} and \eqref{e:overview-goal-compare-step1} imply
  \[
    \bbE[X_N^2] \le 1 + O(N^{-1/3})\,,
  \]
  and that $\theta = c / c_0$.
  By H\"older's inequality,
  \[
    \bbE \lt[
      X_N^2 \la \exp(cN^{1/3}|R_{1,2}|) \ra
    \rt]
    \le \bbE \lt[
      X_N^2 \la \exp(c_0N^{1/3}|R_{1,2}|) \ra
    \rt]^\theta
    \bbE [X_N^2]^{1-\theta}
    \le (1 + O(N^{-1/3})) 2^\theta
    \le 1.5\,,
  \]
  provided $c$ is sufficiently small.
  Then the right-hand side of \eqref{e:A-a0-b0} is bounded by $(0.9)^{-2} \cdot 1.5 \le 1.9$.
\end{proof}
\end{lem}
\begin{lem}
  There exists an absolute constant $C>0$ such that for all $0\le k\le n-1$, we have $B(a_k,b_{k+1}) \le CN^{-c(k+1)}$.
\begin{proof}
  We estimate
  \[
    B(a_k,b_{k+1}) \le N^{b_{k+1}} \bbP(X_N \le N^{-a_k})\,.
  \]
  If $k=0$, then $N^{-a_k} = 0.9$.
  Theorem~\ref{t:compare-cube-sphere} implies $\bbP(X_N \le 0.9) \lesssim N^{-1/3}$, so
  \[
    B(a_0,b_1) \lesssim N^{b_1 - \fr13} = N^{-c}\,,
  \]
  as desired.
  If $k\ge 1$, then Lemmas~\ref{l:fe-lower-tail} and \ref{l:fe-upper-tail} imply
  \balnn
    \bbP(X_N \le N^{-a_k})
    &\le \bbP\lt(F_N \le m - \fr{a_k}{2} \log N\rt)
    + \bbP\lt(F^\sph_N \ge m + \fr{a_k}{2} \log N\rt) \\
    &\le 2\exp\lt(- \fr{a_k^2 \log N}{4}\rt) + \exp\lt(-\lt(\fr{a_k}{2} - \fr{1}{12}\rt)_+ \log N\rt) \\
    &\le 3\exp\lt(-\min\lt[\fr{a_k^2}{4}, \lt(\fr{a_k}{2} - \fr{1}{12}\rt)_+\rt] \log N \rt)\,.
  \ealnn
  By Lemma~\ref{l:sequences}, we have $a_k \ge a_1 \ge 2$, which implies
  \[
    \min\lt[\fr{a_k^2}{4}, \lt(\fr{a_k}{2} - \fr{1}{12}\rt)_+\rt]
    = \fr{a_k}{2} - \fr{1}{12}
    = b_{k+1} + c(k+1)\,.
  \]
  Thus
  \[
    B(a_k,b_{k+1}) \le N^{b_{k+1}} \cdot 3N^{-b_{k+1} - c(k+1)} = 3N^{-c(k+1)}\,. \qedhere
  \]
\end{proof}
\end{lem}
\begin{lem}\label{l:A-bd}
  For all $1\le k\le n-1$, we have $A(a_k,b_k) \le 2N^{-ck}$.
\begin{proof}
  We estimate
  \[
    A(a_k,b_k) \le N^{2a_k} \bbE \lt[
      X_N^2 \la \exp(cN^{1/3}|R_{1,2}|) \bone\{cN^{1/3}|R_{1,2}| \ge b_k \log N\} \ra
    \rt]\,.
  \]
  We can bound
  \balnn
    \bone\{cN^{1/3}|R_{1,2}| \ge b_k \log N\}
    &\le \exp\lt((c_0-c) N^{1/3}|R_{1,2}| - (\theta^{-1} - 1) b_k \log N\rt) \\
    &= N^{-(\theta^{-1} - 1) b_k} \exp\lt((c_0-c) N^{1/3}|R_{1,2}|\rt)\,,
  \ealnn
  so that \eqref{e:reweighted-exp-moment} implies
  \[
    A(a_k,b_k) 
    \le N^{2a_k - (\theta^{-1} - 1)b_k} \bbE \lt[
      X_N^2 \la \exp(c_0N^{1/3}|R_{1,2}|) \ra
    \rt]
    \le 2N^{2a_k - (\theta^{-1} - 1)b_k}
    = 2N^{-ck}\,. \qedhere
  \]
\end{proof}
\end{lem}

\begin{proof}[Proof of Theorem~\ref{t:overlap}\ref{i:overlap-ub}]
  Set $c>0$ small enough that Lemmas~\ref{l:sequences} and \ref{l:A0-term} hold.
  By Lemmas~\ref{l:sequences-decomp}--\ref{l:A-bd},
  \[
    \bbE \la \exp(cN^{1/3} |R_{1,2}|) \ra
    \le 1.9 + \sum_{k=1}^n CN^{-ck} + \sum_{k=1}^{n-1} 2N^{-ck}
    \le 2\,. \qedhere
  \]
\end{proof}

\begin{proof}[Proof of Lemma~\ref{l:sequences}]
  An elementary check shows that if we take $c$ small enough, then 
  \baln\label{e:sequences-init}
    \min(a_1-2,a_2-a_1,b_1,b_2-b_1) &> 0\,, \\
    \label{e:sequences-cvx-init}
    \min(a_3-2a_2+a_1,b_3-2b_2+b_1) &> 0\,.
  \ealn
  The recursions in Definition~\ref{d:a-b-recursion} imply that for all $k\ge 1$,
  \balnn
    a_{k+1} &= \fr{1}{2\theta} \lt[(1-\theta) \lt(\fr{a_k}{2} - \fr{1}{12}\rt) - c(k+1)\rt]\,, &
    b_{k+1} &= \fr{\theta^{-1}-1}{4} b_k - \fr{1}{12} - c\lt(\fr54 k + 1\rt)\,.
  \ealnn
  Taking two finite differences yields
  \balnn
    a_{k+3} - 2a_{k+2} + a_{k+1} &= \fr{1-\theta}{4\theta} (a_{k+2} - 2a_{k+1} + a_k)\,, \\
    b_{k+3} - 2b_{k+2} + b_{k+1} &= \fr{1-\theta}{4\theta} (b_{k+2} - 2b_{k+1} + b_k)\,,
  \ealnn
  for all $k\ge 1$.
  Together with \eqref{e:sequences-cvx-init} this implies that for all $k\ge 1$,
  \beq\label{e:sequences-cvx}
    \min(a_{k+2} - 2a_{k+1} + a_k, b_{k+2} - 2b_{k+1} + b_k) > 0\,.
  \eeq
  Combined with \eqref{e:sequences-init}, this implies $(a_k)_{k\ge 1}$ and $(b_k)_{k\ge 1}$ are positive-valued increasing sequences tending to $\infty$, where $a_1 \ge 2$.
  Since $a_0 = o_N(1)$ is smaller than $a_1$ for sufficiently large $N$, and $b_0 = 0$, the lemma's conclusion follows.
\end{proof}

\subsection{Variance of spherical model free energy}

Finally we prove Corollary~\ref{c:sphere}\ref{i:sphere-var}.
Define the event $\cE = \{|F_N - F^\sph_N| \le N^{-1/9}\}$.
Then Theorem~\ref{t:compare-cube-sphere} implies
\beq\label{e:cE-prob}
  \bbP(\cE^c) = \bbP\lt(X_N - 1 \not\in [\exp(-N^{-1/9}) - 1, \exp(N^{-1/9}) - 1]\rt) \lesssim N^{-1/9}\,.
\eeq
The proof of Corollary~\ref{c:sphere}\ref{i:sphere-var} will be based on the variance decompositions
\baln\label{e:FN-variance-decomp}
  \Var(F_N) &= \bbP(\cE) \Var(F_N | \cE) + \bbP(\cE^c) \Var(F_N | \cE^c) + \Var(\bbE[F_N | \bone\{\cE\}])\,, \\
  \label{e:FNsph-variance-decomp}
  \Var(F^\sph_N) &= \bbP(\cE) \Var(F^\sph_N | \cE) + \bbP(\cE^c) \Var(F^\sph_N | \cE^c) + \Var(\bbE[F^\sph_N | \bone\{\cE\}])\,.
\ealn
\begin{lem}\label{l:cEc-variance}
  We have 
  \[
    \max\lt(
      \bbP(\cE^c) \Var(F_N | \cE^c),
      \bbP(\cE^c) \Var(F^\sph_N | \cE^c)
    \rt) \lesssim N^{-1/10}\,.
  \]
\begin{proof}
  We give the proof for $F_N$, as the proof for $F^\sph_N$ is identical.
  Recall that for any random variable $X$ with finite second moment, $\Var(X) = \inf_{x\in \bbR} \bbE[(X-x)^2]$.
  By this fact and H\"older's inequality,
  \baln
    \bbP(\cE^c) \Var(F_N | \cE^c)
    &\le \bbP(\cE^c) \bbE[(F_N - m)^2 \,|\,\cE^c\,] \\
    &= \bbE\lt[\bone\{\cE^c\} (F_N - m)^2 \rt]
    \le \bbP(\cE^c)^{10/11} \bbE[(F_N - m)^{22}]^{1/11}\,.
  \ealn
  Then \eqref{e:cE-prob} gives 
  \[
    \bbP(\cE^c)^{10/11} \lesssim N^{-10/99}\,.
  \]
  The tail bounds from Lemmas~\ref{l:fe-lower-tail} and \ref{l:fe-upper-tail}, and a routine tail integration argument, imply
  \[
    \bbE[(F_N - m)^{22}]
    \le \polylog(N)\,.
  \]
  For the proof for $F^\sph_N$, the last estimate can be proved by replacing Lemma~\ref{l:fe-lower-tail} with Lemma~\ref{l:fe-sph-lower-tail}.
  The result follows for sufficiently large $N$ because $10/99 > 1/10$.
\end{proof}
\end{lem}
\begin{lem}\label{l:mean-variance-contribution}
  We have 
  \[
    \max\lt(
      \Var(\bbE[F_N | \bone\{\cE\}]),
      \Var(\bbE[F^\sph_N | \bone\{\cE\}])
    \rt) \lesssim N^{-1/10}\,.
  \]
\begin{proof}
  Again we only give the proof for $F_N$.
  Lemmas~\ref{l:fe-lower-tail} and \ref{l:fe-upper-tail} and a tail integration argument imply $\bbE[F_N]$ exists and
  \beq\label{e:EF-bound}
    |\bbE[F_N] - m| \lesssim \log N\,.
  \eeq
  Since
  \[
    \bbE[F_N] = \bbP(\cE) \bbE[F_N | \cE] + \bbP(\cE^c) \bbE[F_N | \cE^c]\,, 
  \]
  a routine calculation shows
  \[
    \Var(\bbE[F_N | \bone\{\cE\}])
    = \fr{\bbP(\cE^c)}{\bbP(\cE)} (\bbE[F_N | \cE^c] - \bbE[F_N])^2
    \lesssim \bbP(\cE^c) (\bbE[F_N | \cE^c] - \bbE[F_N])^2\,.
  \]
  We then estimate using Jensen's and H\"older's inequalities
  \balnn
    \bbP(\cE^c) (\bbE[F_N | \cE^c] - \bbE[F_N])^2
    &\le \bbP(\cE^c) \bbE\lt[
      (F_N - \bbE[F_N])^2 | \cE^c
    \rt] \\
    &= \bbE\lt[
      \bone\{\cE^c\} (F_N - \bbE[F_N])^2
    \rt]
    \le \bbP(\cE^c)^{10/11} \bbE[(F_N - \bbE[F_N])^{22}]^{1/11}\,.
  \ealnn
  The result follows similarly to Lemma~\ref{l:cEc-variance}, as Lemmas~\ref{l:fe-lower-tail} and \ref{l:fe-upper-tail}, the estimate \eqref{e:EF-bound}, and a tail integration argument imply
  \[
    \bbE[(F_N - \bbE[F_N])^{22}]^{1/11} \le \polylog(N)\,.
  \]
  For $F^\sph_N$, replace all invocations of Lemma~\ref{l:fe-lower-tail} with Lemma~\ref{l:fe-sph-lower-tail}.
\end{proof}
\end{lem}

\begin{proof}[Proof of Corollary~\ref{c:sphere}\ref{i:sphere-var}]
  We will show the stronger estimate
  \beq\label{e:sphere-var}
    \Var(F^\sph_N) = \Var(F_N) + O(N^{-1/10})\,.
  \eeq
  Plugging \eqref{e:cE-prob} and Lemmas~\ref{l:cEc-variance}--\ref{l:mean-variance-contribution} into the variance decomposition \eqref{e:FN-variance-decomp} yields
  \[
    \Var(F_N) = (1 - O(N^{-1/9})) \Var(F_N | \cE) + O(N^{-1/10})\,.
  \]
  Theorem~\ref{t:main}\ref{i:main-var}, this implies 
  \beq\label{e:var-FN-on-cE}
    \Var(F_N | \cE) = \Var(F_N) + O(N^{-1/10}) = \fr16 \log N + O(1)\,.
  \eeq
  On the event $\cE$, we have $|F_N - F^\sph_N| \le N^{-1/9}$.
  Write
  \[
    \Delta = (F_N - \bbE[F_N | \cE]) - (F^\sph_N - \bbE[F^\sph_N | \cE])\,,
  \]
  so that $|\Delta| \le 2N^{-1/9}$ on $\cE$.
  Thus, by Cauchy-Schwarz,
  \balnn
    \Var(F_N | \cE) &= \bbE\lt[
      (F_N - \bbE[F_N | \cE])^2 \,|\, \cE
    \rt] = \bbE\lt[
      (F^\sph_N - \bbE[F^\sph_N | \cE] + \Delta)^2 \,|\, \cE
    \rt] \\
    &\le (1 + N^{-1/9}) \bbE\lt[
      (F^\sph_N - \bbE[F^\sph_N | \cE])^2 \,|\, \cE
    \rt] + (1 + N^{1/9}) \bbE\lt[
      \Delta^2 \,|\, \cE
    \rt] \\
    &= (1 + N^{-1/9}) \Var(F^\sph_N | \cE) + O(N^{-1/9})\,.
  \ealnn
  Reversing the roles of $F_N$ and $F^\sph_N$ shows
  \[
    \Var(F^\sph_N | \cE) \le (1 + N^{-1/9}) \Var(F_N | \cE) + O(N^{-1/9})\,.
  \]
  Together with \eqref{e:var-FN-on-cE} this ensures
  \[
    \Var(F^\sph_N | \cE) = \Var(F_N) + O(N^{-1/10})\,.
  \]
  Plugging this estimate, \eqref{e:cE-prob}, and Lemmas~\ref{l:cEc-variance}--\ref{l:mean-variance-contribution} into the variance decomposition \eqref{e:FNsph-variance-decomp} yields \eqref{e:sphere-var}.
\end{proof}

\bibliographystyle{alpha}
\bibliography{bib}

@article{aspelmeier2008free,
  author={Aspelmeier, Timo},
  title={Free-energy fluctuations and chaos in the {S}herrington--{K}irkpatrick model},
  journal={Phys. Rev. Lett.}, 
  volume={100}, 
  pages={117205}, 
  year={2008}
}

@article{parisi2009phase,
  author={Parisi, Giorgio and Rizzo, Tommaso},
  title={Phase diagram and large deviations in the free-energy of mean-field spin-glasses},
  journal={Phys. Rev. B},
  volume={79},
  pages={134205}, 
  year={2009}
}

@article{chen2019order,
  author={Chen, Wei-Kuo and Lam, Wai-Kit},
  title={Order of fluctuations of the free energy in the {SK} model at critical temperature},
  journal={ALEA Lat. Am. J. Probab. Math. Stat.},
  volume={16},
  number={1},
  pages={809--816},
  year={2019}
}

@article{dey2026fluctuations,
  author={Dey, Partha S. and Kang, Taegu},
  title={Fluctuations for the {S}herrington--{K}irkpatrick spin glass model near the critical temperature},
  journal={arXiv preprint arXiv:2603.05636},
  year={2026},
}

@article{landon2022free,
  author={Landon, Benjamin},
  title={Free energy fluctuations of the two-spin spherical {SK} model at critical temperature},
  journal={J. Math. Phys.}, 
  volume={63},
  number={3},
  pages={Paper No. 033301},
  year={2022}
}

@article{johnstone2024spin,
  author={Johnstone, Iain M. and Klochkov, Yegor and Onatski, Alexei and Pavlyshyn, Damian},
  title={Spin glass to paramagnetic transition and triple point in spherical {SK} model},
  journal={J. Stat. Phys.},
  volume={191},
  number={99},
  year={2024}
}

@article{collins2025fluctuations,
  author={Collins{-}Woodfin, Elizabeth W. and Le, Han Gia},
  title={Free energy fluctuations of the bipartite spherical {SK} model at critical temperature},
  journal={Ann. Henri Poincar\'e},
  volume={26},
  pages={1087--1147},
  year={2025}
}

@article{collins2025order,
  author={Collins{-}Woodfin, Elizabeth W. and Le, Han Gia},
  title={Order of fluctuations of the free energy in the positive semi-definite {MSK} model at critical temperature},
  journal={arXiv preprint arXiv:2501.11732},
  year={2025}
}

@article{prodromidis2026distribution,
  author={Prodromidis, Kyprianos-Iason and Sly, Allan},
  title={Distribution of the magnetization of the critical {I}sing model on sparse random graphs},
  journal={arXiv preprint arXiv:2603.28702},
  year={2026}
}

@article{aizenman1987some,
  author       = {Aizenman, Michael and Lebowitz, Joel L. and Ruelle, David},
  title        = {Some rigorous results on the {S}herrington--{K}irkpatrick spin glass model},
  journal      = {Comm. Math. Phys.},
  volume       = {112},
  pages        = {3--20},
  year         = {1987},
}

@article{comets1995sherrington,
  author       = {Comets, Francis and Neveu, Jacques},
  title        = {The {S}herrington--{K}irkpatrick model of spin glasses and stochastic calculus: the high temperature case},
  journal      = {Comm. Math. Phys.},
  volume       = {166},
  pages        = {549--564},
  year         = {1995},
}

@book{talagrand2010mean,
  author       = {Talagrand, Michel},
  title        = {Mean Field Models for Spin Glasses: Volume I: Basic Examples},
  publisher    = {Springer},
  volume       = {54},
  year         = {2010},
}

@book{talagrand2011mean2,
  author       = {Talagrand, Michel},
  title        = {Mean Field Models for Spin Glasses. Volume II: Advanced Replica-Symmetry and Low Temperature},
  publisher    = {Springer},
  volume       = {55},
  year         = {2011},
}

@article{talagrand2006parisi,
  author       = {Talagrand, Michel},
  title        = {The {P}arisi formula},
  journal      = {Ann. Math.},
  volume       = {163},
  number       = {1},
  pages        = {221--263},
  year         = {2006},
}

@article{baik2016fluctuations,
  author       = {Baik, Jinho and Lee, Ji Oon},
  title        = {Fluctuations of the free energy of the spherical {S}herrington--{K}irkpatrick model},
  journal      = {J. Stat. Phys.},
  volume       = {165},
  pages        = {185--224},
  year         = {2016},
}

@book{chatterjee2014superconcentration,
  author       = {Chatterjee, Sourav},
  title        = {Superconcentration and Related Topics},
  publisher    = {Springer},
  volume       = {15},
  year         = {2014},
}

@article{chatterjee2009disorder,
  author       = {Chatterjee, Sourav},
  title        = {Disorder chaos and multiple valleys in spin glasses},
  journal      = {arXiv preprint arXiv:0907.3381},
  year         = {2009}
}

@article{sherrington1975solvable,
  author       = {Sherrington, David and Kirkpatrick, Scott},
  title        = {Solvable model of a spin-glass},
  journal      = {Phys. Rev. Lett.},
  volume       = {35},
  number       = {26},
  pages        = {1792},
  year         = {1975},
}

@article{parisi1979infinite,
  author       = {Parisi, Giorgio},
  title        = {Infinite number of order parameters for spin-glasses},
  journal      = {Phys. Rev. Lett.},
  volume       = {43},
  number       = {23},
  pages        = {1754},
  year         = {1979},
}

@article{parisi1983order,
  author       = {Parisi, Giorgio},
  title        = {Order parameter for spin-glasses},
  journal      = {Phys. Rev. Lett.},
  volume       = {50},
  number       = {24},
  pages        = {1946},
  year         = {1983},
}

@article{panchenko2013parisi,
  author       = {Panchenko, Dmitry},
  title        = {The {P}arisi ultrametricity conjecture},
  journal      = {Ann. Math.},
  volume       = {177},
  number       = {1},
  pages        = {383--393},
  year         = {2013},
}

@book{mezard1987spin,
  author       = {M{\'e}zard, Marc and Parisi, Giorgio and Virasoro, Miguel A.},
  title        = {Spin Glass Theory and Beyond: An Introduction to the Replica Method and Its Applications},
  publisher    = {World Scientific},
  volume       = {9},
  year         = {1987},
}

@article{ruelle1987mathematical,
  author       = {Ruelle, David},
  title        = {A mathematical reformulation of {D}errida's {REM} and {GREM}},
  journal      = {Comm. Math. Phys.},
  volume       = {108},
  number       = {2},
  pages        = {225--239},
  year         = {1987},
}

@article{ghirlanda1998general,
  title={General properties of overlap probability distributions in disordered spin systems. {T}owards {P}arisi ultrametricity},
  author={Ghirlanda, Stefano and Guerra, Francesco},
  journal={J. Phys. A},
  volume={31},
  number={46},
  pages={9149--9155},
  year={1998}
}

@article{aizenman2003extended,
  author       = {Aizenman, Michael and Sims, Robert and Starr, Shannon L.},
  title        = {Extended variational principle for the {S}herrington--{K}irkpatrick spin-glass model},
  journal      = {Phys. Rev. B},
  volume       = {68},
  number       = {21},
  pages        = {214403},
  year         = {2003},
}

@article{guerra2003broken,
  author       = {Guerra, Francesco},
  title        = {Broken replica symmetry bounds in the mean field spin glass model},
  journal      = {Comm. Math. Phys.},
  volume       = {233},
  number       = {1},
  pages        = {1--12},
  year         = {2003},
}

@article{talagrand2006spherical,
  author       = {Talagrand, Michel},
  title        = {Free energy of the spherical mean field model},
  journal      = {Probab. Theory Rel. Fields},
  volume       = {134},
  pages        = {339--382},
  year         = {2006},
}

@article{dey2021fluctuation,
  title={Fluctuation results for multi-species {S}herrington--{K}irkpatrick model in the replica symmetric regime},
  author={Dey, Partha S. and Wu, Qiang},
  journal={J. Stat. Phys.},
  volume={185},
  number={3},
  pages={22},
  year={2021}
}

@article{baik2020free,
  title={Free energy of bipartite spherical {S}herrington--{K}irkpatrick model},
  author={Baik, Jinho and Lee, Ji Oon},
  journal={Ann. Inst. Henri Poincar\'e Probab. Stat.},
  volume={56},
  number={4},
  pages={2897--2934},
  year={2020}
}

@article{robinson1992almost,
  title={Almost all cubic graphs are {H}amiltonian},
  author={Robinson, Robert W. and Wormald, Nicholas C.},
  journal={Rand. Struct. Alg.},
  volume={3},
  number={2},
  pages={117--125},
  year={1992}
}

@article{robinson1994almost,
  title={Almost all regular graphs are {H}amiltonian},
  author={Robinson, Robert W. and Wormald, Nicholas C.},
  journal={Rand. Struct. Alg.},
  volume={5},
  number={2},
  pages={363--374},
  year={1994}
}

@article{mossel2009hardness,
  title={On the hardness of sampling independent sets beyond the tree threshold},
  author={Mossel, Elchanan and Weitz, Dror and Wormald, Nicholas},
  journal={Probab. Theory Rel. Fields},
  volume={143},
  number={3},
  pages={401--439},
  year={2009}
}

@article{kemkes2010chromatic,
  title={On the chromatic number of random $d$-regular graphs},
  author={Kemkes, Graeme and P{\'e}rez-Gim{\'e}nez, Xavier and Wormald, Nicholas},
  journal={Adv. Math.},
  volume={223},
  number={1},
  pages={300--328},
  year={2010}
}

@inproceedings{abbe2022proof,
  title={Proof of the contiguity conjecture and lognormal limit for the symmetric perceptron},
  author={Abbe, Emmanuel and Li, Shuangping and Sly, Allan},
  booktitle={Proc. 62nd FOCS},
  pages={327--338},
  year={2022}
}

@article{bencs2025zeros,
  title={On zeros and algorithms for disordered systems: mean-field spin glasses},
  author={Bencs, Ferenc and Huang, Brice and Lee, Daniel Z. and Liu, Kuikui and Regts, Guus},
  journal={arXiv preprint arXiv:2507.15616},
  year={2025}
}

@article{cooper1996perfect,
  title={Perfect matchings in random $r$-regular, $s$-uniform hypergraphs},
  author={Cooper, Colin and Frieze, Alan and Molloy, Michael and Reed, Bruce},
  journal={Combin. Probab. Comp.},
  volume={5},
  number={1},
  pages={1--14},
  year={1996}
}

@article{coja2026fluctuations,
  title={Fluctuations of the {I}sing free energy on {E}rd{\H{o}}s--{R}\'{e}nyi graphs},
  author={Coja{-}Oghlan, Amin and Kaaser, Dominik and Rolvien, Maurice and Zakharov, Pavel and Zampetakis, Kostas},
  journal={arXiv preprint arXiv:2601.08590},
  year={2026}
}

@article{fabian2021ising,
  title={The {I}sing antiferromagnet in the replica symmetric phase},
  author={Fabian, Christian and Loick, Philipp},
  journal={arXiv preprint arXiv:2103.09775},
  year={2021}
}

@article{chen2023gaussian,
  title={A gaussian convexity for logarithmic moment generating functions with applications in spin glasses},
  author={Chen, Wei-Kuo},
  journal={arXiv preprint arXiv:2311.08351},
  year={2023}
}

@article{baik2005phase,
  author  = {Baik, Jinho and Ben\space{}Arous, G{\'e}rard and P{\'e}ch{\'e}, Sandrine},
  title   = {Phase transition of the largest eigenvalue for nonnull complex sample covariance matrices},
  journal = {Ann. Probab.},
  volume  = {33},
  number  = {5},
  pages   = {1643--1697},
  year    = {2005}
}

@article{peche2006largest,
  author  = {P{\'e}ch{\'e}, Sandrine},
  title   = {The largest eigenvalue of small rank perturbations of {H}ermitian random matrices},
  journal = {Probab. Theory Rel. Fields},
  volume  = {134},
  pages   = {127--173},
  year    = {2006}
}

@article{bloemendal2013limits,
  author  = {Bloemendal, Alex and Vir{\'a}g, B{\'a}lint},
  title   = {Limits of spiked random matrices {I}},
  journal = {Probab. Theory Rel. Fields},
  volume  = {156},
  pages   = {795--825},
  year    = {2013}
}

@article{bloemendal2016limits,
  author  = {Bloemendal, Alex and Vir{\'a}g, B{\'a}lint},
  title   = {Limits of spiked random matrices {II}},
  journal = {Ann. Probab.},
  volume  = {44},
  number  = {4},
  pages   = {2726--2769},
  year    = {2016}
}

@article{ErdosYauYin2012Rigidity,
  author  = {Erd{\H{o}}s, L{\'a}szl{\'o} and Yau, Horng-Tzer and Yin, Jun},
  title   = {Rigidity of Eigenvalues of Generalized {W}igner Matrices},
  journal = {Adv. Math.},
  volume   = {229},
  number   = {3},
  pages    = {1435--1515},
  year     = {2012}
}

@article{schertzer2026order,
  title = {The order of free energy fluctuations in the critical {S}herrington--{K}irkpatrick model revisited},
  author = {Schertzer, Adrien},
  journal = {arXiv preprint arXiv:2606.21360},
  year = {2026}
}

@book{anderson2010introduction,
  title={An Introduction to Random Matrices},
  author={Anderson, Greg W. and Guionnet, Alice and Zeitouni, Ofer},
  number={118},
  year={2010},
  publisher={Cambridge university press}
}

@article{landon2022fluctuations,
  title={Fluctuations of the overlap at low temperature in the $2$-spin spherical {SK} model},
  author={Landon, Benjamin and Sosoe, Philippe},
  journal={Ann. Inst. Henri Poincar\'e Probab. Stat.},
  volume={58},
  number={3},
  pages={1426--1459},
  year={2022}
}

@book{FK94,
  author    = {Jacques Faraut and Adam Kor{\'a}nyi},
  title     = {Analysis on Symmetric Cones},
  series    = {Oxford Mathematical Monographs},
  publisher = {Clarendon Press},
  address   = {Oxford},
  year      = {1994},
  isbn      = {9780198534778}
}

@book{widder1946laplace,
  author    = {David Vernon Widder},
  title     = {The Laplace Transform},
  series    = {Princeton Mathematical Series},
  volume    = {6},
  publisher = {Princeton University Press},
  address   = {Princeton, NJ},
  year      = {1946}
}

@article{du2026overlap,
  title={Overlap distribution of the critical {S}herrington--{K}irkpatrick model},
  author={Du, Hang and Huang, Brice},
  journal={arXiv preprint arXiv:2608.08752},
  year={2026}
}

@article{du2026critical,
  title={Free energy fluctuations of symmetric perceptron models up to and including criticality},
  author={Du, Hang and Gong, Shuyang and Huang, Brice and Li, Zhangsong},
  journal={arXiv preprint arXiv:2610.31736},
  year={2026}
}

@article{cheng2026critical,
  title={Critical-Window Fluctuations and Disorder Universality for the {S}herrington--{K}irkpatrick Model},
  author={Cheng, Yu and Liu, Song-Hao and Shao, Qi-Man and Xu, Jing-Yu},
  journal={arXiv preprint arXiv:2609.07446},
  year={2026}
}

@article{aronow2026quantitative,
  title={Quantitative {P}arisi formulas and fluctuations in the {S}herrington--{K}irkpatrick model},
  author={Aronow, Peter M and Lopatto, Patrick},
  journal={arXiv preprint arXiv:2609.09103},
  year={2026}
}

@misc{openai2026low,
  author = {OpenAI}, 
  title = {The low-temperature {S}herrington–-{K}irkpatrick free-energy limiting law},
  year = {2026},
  url = {https://github.com/openai/math/blob/main/preprints/The-low-temperature-Sherrington-Kirkpatrick-free-energy-limiting-law-September-24-2026/The-low-temperature-Sherrington-Kirkpatrick-free-energy-limiting-law-September-24-2026.pdf}
}

\end{document}